\documentclass{amsart}
\usepackage{amsfonts,amsthm,amsmath}
\usepackage{amssymb}
\usepackage{amscd}
\usepackage[utf8]{inputenc}
\usepackage[a4paper]{geometry}
\usepackage{enumerate}
\usepackage[usenames,dvipsnames]{color}
\usepackage{tikz}
\usepackage{soul}
\usepackage{array}
\usepackage{array,tabularx}

\numberwithin{equation}{section}
\numberwithin{equation}{subsection}

\theoremstyle{plain}

\newtheorem{theorem}[equation]{Theorem}
\newtheorem{lemma}[equation]{Lemma}
\newtheorem{proposition}[equation]{Proposition}

\newtheorem{corollary}[equation]{Corollary}

\theoremstyle{definition}
\newtheorem{example}[equation]{Example}
\newtheorem{remark}[equation]{Remark}

\newtheorem{definition}[equation]{Definition}

\newtheorem{bekezdes}[equation]{}

\newcommand{\tir}{\tilde{r}}
\newcommand{\tZ}{\widetilde{Z}}

\newcommand{\tomega}{\widetilde{\omega}}

\def\C{\mathbb C}
\def\Q{\mathbb Q}

\def\Z{\mathbb Z}

\newcommand{\cale}{{\mathcal E}}

\newcommand{\caly}{{\mathcal Y}}
\newcommand{\calx}{{\mathcal X}}\newcommand{\calz}{{\mathcal Z}}
\newcommand{\calv}{{\mathcal V}}

\newcommand{\calm}{{\mathcal M}}

\newcommand{\cali}{{\mathcal I}}

\newcommand{\calf}{{\mathcal F}}
\newcommand{\calO}{{\mathcal O}}

\newcommand{\calS}{{\mathcal S}}

\newcommand{\calL}{\mathcal{L}}

\newcommand{\tX}{\widetilde{X}}

\newcommand{\cO}{{\mathcal O}}
\newcommand{\bP}{{\mathbb P}}

\newcommand{\bC}{{\mathbb C}}

\newcommand{\eca}{{\rm ECa}}
\newcommand{\pic}{{\rm Pic}}

\newcommand{\bt}{{\mathbf t}}

\newcommand{\bZ}{{\mathbb{Z}}}
\newcommand{\bQ}{{\mathbb{Q}}}

\newcommand{\omegl}{\lambda}

\newcommand{\omeg}{\varsigma}
\newcommand{\tomeg}{\widetilde{\omeg}}

\author{J\'anos Nagy}
\address{Central European University, Dept. of Mathematics,  Budapest, Hungary}
\email{nagy\textunderscore janos@phd.ceu.edu}

\author{Andr\'as N\'emethi}
\address{Alfr\'ed R\'enyi Institute of Mathematics,
Hungarian Academy of Sciences,
Re\'altanoda utca 13-15, H-1053, Budapest, Hungary \newline
 \hspace*{4mm} ELTE - University of Budapest, Dept. of Geometry, Budapest, Hungary \newline \hspace*{4mm}
BCAM - Basque Center for Applied Math.,
Mazarredo, 14 E48009 Bilbao, Basque Country – Spain}
\email{nemethi.andras@renyi.mta.hu }

\title{The Abel map for surface singularities \\
II. Generic analytic structure
}

\begin{document}



\begin{abstract}
We study the analytic and topological invariants
associated with complex normal surface singularities.
Our goal is to provide topological formulae for several discrete analytic invariants
whenever the analytic structure is generic  (with respect to a fixed topological type),
under the condition that the link is  a rational homology sphere.
The list of analytic invariants include: the geometric genus, the cohomology of certain
natural line bundles, the cohomology of their restrictions on effective cycles
(supported on the exceptional curve of a resolution),
the cohomological cycle of natural line bundles,
the multivariable Hilbert and Poincar\'e series associated with the divisorial filtration,
the analytic semigroup, the maximal ideal cycle.

The first part contains the definition of `generic structure' based on the work of Laufer
\cite{LaDef1}.
The second technical ingredient  is the Abel map developed in \cite{NNI}.

The results can be compared with certain parallel statements from
the Brill--Noether theory and from the theory of Abel map associated with projective
smooth curves, though the tools and machineries are  very different.
\end{abstract}

\maketitle

\linespread{1.2}


\pagestyle{myheadings} \markboth{{\normalsize  J. Nagy, A. N\'emethi}} {{\normalsize Abel maps }}


\section{Introduction}\label{s:intr}
\subsection{} Our major objects in this note are the analytic and topological invariants
associated with complex normal surface singularity germs.
Our goal is to provide topological formulae for several discrete analytic invariants
whenever the analytic structure is generic  (with respect to a fixed topological type).
Regarding this problem very little is known in the present literature.
The type of formulae of the topological characterizations of the present article are
totally new, as well as the methods (based on the newly created theory of Abel map).

In order to formulate the invariants and the topological characterizations we need some notation.
Let $\tX\to X$ be a good resolution with irreducible exceptional curves $\{E_v\}_{v\in\calv}$,
with resolution graph $\Gamma$, negative definite
intersection lattice $L=H_2(\tX,\Z)$, dual lattice $L'=H^2(\tX,\Z)\simeq H_2(\tX,\partial \tX,\Z)$,
and discriminant group $H=L'/L$ (for details see \ref{ss:2.1}). We assume that the link $M$ of $(X,o)$ is a rational homology sphere, that is, $\Gamma$ is a tree of rational $E_v$'s. In such a case $H=H_1(M,\Z)$ is finite. Usually $Z$ will denote an effective cycle supported on the exceptional curve $E$. $L'$ is also the target of the surjective first Chern class map $c_1:{\rm Pic}(\tX)\to L'$, set
 $c_1^{-1}(l')={\rm Pic}^{l'}(\tX)$. For any Chern class one defines
the `natural line bundle' $\calO_{\tX}(l')\in {\rm Pic}^{l'}(\tX)$, and its restrictions
$\calO_Z(l')$, cf. \ref{ss:natline}.

In the sequel we fix a topological type, that is, a resolution graph. The topological invariants are
read from $\Gamma$, or equivalently, from $L$. The most elementary one is the
`Riemann--Roch' expression $\chi:L'\to \Q$ given by $\chi(l'):=-(l',l'-Z_K)/2$, where $Z_K\in L'$
is the anticanonical cycle defined combinatorially by the adjunction formulae, cf. \ref{ss:2.1}.

 The list of analytic invariants, associated with a generic analytic type
  (with respect to the fixed graph),
 which are described in the present article topologically
  are the following:
 $h^1(\calO_Z)$, $h^1(\calO_Z(l'))$ (with certain restriction on the Chern class $l'$),
  --- this last one applied  for $Z\gg 0$ provides  $h^1(\calO_{\tX})$
 and  $h^1(\calO_{\tX}(l')$) too ---, the cohomological cycle of natural line bundles,
the multivariable Hilbert and Poincar\'e series associated with the divisorial filtration,
the analytic semigroup, the maximal ideal cycle.
See \cite{CDGPs,CDGEq,Lipman,Nfive,NPS,NCL,Ok,MR}
for the definitions and relationships between them. Here some definitions will be recalled in section \ref{s:appli}.

Surprisingly, in all the
 topological characterization we need to use merely $\chi$,
 however,  it is really remarkable  the level of
complexity and subtlety  of the
 combinatorial expressions/invariants carried by this `simple' (?) quadratic function.
Definitely, this  can happen due to the fact that we work over the lattices $L$ and $L'$,
and the position of the lattice points with respect to the level sets of $\chi$ play
the key role.
It is a real challenge now to interpret these expressions in terms of lattice cohomology \cite{lattice,NJEMS} or other topological 3--manifold invariants.

\vspace{2mm}

\noindent {\bf Theorem A.}
{\em Fix a resolution graph and assume that the analytic type of  $\tX$ is generic. Then
the following identities hold:\\
(a) For any effective cycle $Z\in L_{>0}$
\begin{equation*}
h^1(\calO_Z) = 1-\min_{0< l \leq Z,l\in L}\{\chi(l)\}.
\end{equation*}
(b) If $l'=\sum_{v\in \calv}l'_vE_v \in L'$ satisfies
$l'_v <0$ for any $E_v$ in the support of $Z$ then
\begin{equation*}
h^1(Z,\calO_Z(l'))=\chi(-l')-\min _{0\leq l\leq Z, l\in L}\, \{\chi(-l'+l)\}.
\end{equation*}
(For a characterization valid for more general Chern classes $l'$ see section \ref{s:appli}.)\\
(c) If $p_g(X,o)=h^1(\tX,\calO_{\tX})$ is the geometric genus of $(X,o)$ then
\begin{equation*}
p_g(X,o)= 1-\min_{l\in L_{>0}}\{\chi(l)\} =-\min_{l\in L}\{\chi(l)\}+\begin{cases}
1 & \mbox{if $(X,o)$ is not rational}, \\
0 & \mbox{else}.
\end{cases}
\end{equation*}
(d) More generally, for any $l'\in L'$
\begin{equation*}
h^1(\tX,\calO_{\tX}(l'))=\chi(-l')-\min _{l\in L_{\geq 0}}\, \{\chi(-l'+l)\}+
\begin{cases}
1 & \mbox{if \ $l'\in L_{\geq 0}$ and $(X,o)$ is not rational}, \\
0 & \mbox{else}.
\end{cases}
\end{equation*}
(e) Let $H(\bt)=\sum _{l'\in L'} \mathfrak{h}(l')\bt^{l'}$ be the multivariable equivariant Hilbert series associated with the divisorial filtration.
Write $l'$ as $r_h+l_0$ for some $l_0\in L$ and  $r_h\in L' $ the unique representative of $h=[l']$ in the semi-open cube of $L'$.
Then  $\mathfrak{h}(r_h)=0$ for $l_0=0$. Furthermore, for $l_0>0$ and $h\not=0$
\begin{equation*}
\mathfrak{h}(l')=
\min_{l\in L_{\geq 0}} \{\chi(l'+l)\}-\min_{l\in L_{\geq 0}} \{\chi(r_h+l)\}.
\end{equation*}
For $h=0$ and $l'=l_0>0$
\begin{equation*}
\mathfrak{h}(l_0)=\min_{l\in L_{\geq 0}} \{\chi(l_0+l)\}-\min_{l\in L_{\geq 0}} \{\chi(l)\}+
\begin{cases}
1 & \mbox{if $(X,o)$ is not rational}, \\
0 & \mbox{else}.
\end{cases}
\end{equation*}
(f)
Write the multivariable equivariant Poincar\'e series $P(\bt)=-H(\bt)\cdot \prod_{v\in\calv}
(1-t_v^{-1})$ as $\sum_{l'\in\calS'}\mathfrak{p}(l')\bt^{l'}$.
It is supported in the Lipman (antinef) cone, in particular in $L'_{\geq 0}$.
 Then
$\mathfrak{p}(0)=1$ and for $l'>0$ one has
$$\mathfrak{p}(l')= \sum_{I\subset \calv} (-1)^{|I|+1}\, \min_{ l\in L_{\geq 0}}
\chi(l'+l+E_I).$$
(g) Consider  the   analytic semigroup
$\calS'_{an}:= \{l'\in L' \,:\,\calO_{\tX}(l')\ \mbox{has no  fixed components }\}$. Then
$$\calS'_{an}= \{l'\,:\, \chi(l')<
\chi(l' +l) \ \mbox{for any $l\in L_{>0}$}\}\cup\{0\}.$$
(h)
 Assume that $\Gamma$ is a non--rational graph  and set
$\calm=\{ Z\in L_{>0}\,:\, \chi(Z)=\min _{l\in L}\chi(l)\}$.

Then the unique minimal element of $\calm$ is the cohomological cycle, while the
unique maximal element of $\calm$ is the maximal ideal cycle of $\tX$.
}

\subsection{The Abel map}\label{ss:Abel} The main tool of the present note is the Abel map constructed
and studied in \cite{NNI}. Though in [loc.cit.] we also listed several applications,  the
present note shows its power, its  applicability in a really difficult problem,  with
a priori unexpected answers which become totally natural and motivated from the perspective of this new approach.

Let us recall shortly this object
(for details see \cite{NNI} or \S\ref{s:prel} and \ref{ss:natline} here).
Let  $(X,o)$,   $\tX\to X$,  $L$ and  $L'$ as above.
Then for any effective cycle
$Z$ supported on $E$   and for any (possible) Chern class $l'\in L'$ we consider
the space $\eca^{l'}(Z)$ of effective Cartier
divisors $D$ supported on $Z$, whose associated line bundles $\calO_Z(D)$ have first Chern class $l'$.
Furthermore, we also consider the Abel map  $c^{l'}(Z):\eca^{l'}(Z)\to \pic^{l'}(Z)$, $D\mapsto \calO_Z(D)$.

Using the Abel map,  in \cite[Th. 5.3.1]{NNI} we have shown  that for
any  analytic singularity and resolution with fixed resolution graph,  and for
any $\calL\in \pic^{l'}(Z)$,  one has
$h^1(Z,\calL)\geq \chi(-l')-\min_{0\leq l\leq Z, \, l\in L}\chi(-l'+l)$, and equality holds for
a generic line bundle $\calL_{gen}\in \pic^{l'}(Z)$.
 In particular, for any analytic type, $\calL_{gen}\in \pic^{l'}(Z)$ can be expressed
 combinatorially.
 Now,  the expectation and our guiding principle is the following:
for a generic analytic structure the natural line bundle $\calO_Z(l')$ should have
 the same $h^1$ as the
generic line bundle  $\calL_{gen}\in \pic^{l'}(Z)$ (associated with any analytic structure).
This is the key technical statement  of the note.

\vspace{2mm}

\noindent {\bf Theorem B.} {\it
  Assume that  $\tX$ is generic. Under some (necessary)
  negativity  restriction on the Chern class
  $l'$ (see Theorem \ref{th:CLB1} and Remark \ref{rem:cohGen}(b))
   the following facts hold.

\vspace{1mm}

\noindent (I) \    The following facts are equivalent:

(a) $\calO_Z(l')\in {\rm im}(c^{\tilde{l}})$, where $\calO_Z(l')$ is the natural line bundle
with Chern class $l'$;

(b) $\calL_{gen}\in {\rm im}(c^{\tilde{l}})$, where  $\calL_{gen}$
 is a generic line bundle in $\pic^{\tilde{l}}(Z)$ (that is,
 $c^{\tilde{l}}$ is dominant);

 (c) $\calO_Z(l')\in {\rm im}(c^{\tilde{l}})$,
 and for any $D\in (c^{\tilde{l}})^{-1}(\calO_Z(l'))$ the tangent map
 $T_Dc^{\tilde{l}}: T_D\eca^{\tilde{l}}(Z)\to T_{\calO_Z(l')}\pic^{\tilde{l}}(Z)$ is surjective.

 \vspace{1mm}

 \noindent (II)
$h^i(Z,\calO_Z(l'))=h^i(Z,\calL_{gen})$  for $i=0,1$ and
 for a   generic line bundle
 $\calL_{gen}\in \pic^{\tilde{l}}(Z)$.}

 \vspace{3mm}

 The proof is long and technical, it fills in  all section 5 (the `hard' part
  is (a)$\Rightarrow$(c)).
It uses the explicit description  of tangent map of $c^{l'}$   in terms of  Laufer
duality (integration of forms along divisors, cf. \ref{ss:Integ}).
In this section certain  familiarity with \cite{NNI} might help the reading.

\vspace{2mm}

By this result, if $\tX$ has generic analytic structure, then
the cohomology of natural line bundles can be expressed by the very same topological formula
as $\calL_{gen}$ with the same Chern class. Then all the formulae of Theorem A above follow
directly.

\vspace{2mm}

In the next paragraph we say a few words about `generic analytic type'.

\subsection{Discussion regarding the `generic analytic type'}
Let us comment first what kind of difficulties appear in the definition and study
 of `generic' analytic type.
The point is that for  a fixed topological type the moduli space  of
all analytic structures supported by that fixed topological type, is not yet described in
the literature; hence, we cannot define our generic structure as a generic point of such a space.
Laufer in \cite{Laufer73} characterized those topological types which support only one
analytic type, but about the general cases very little is known.
Usually, generic structures --- when they appeared ---
were introduced   by certain ad-hoc definitions,
or only in particular situations. In a slightly different direction
a remarkable  progress was made by Laufer (see e.g. \cite{LaDef1})
when  he defined {\it local complete deformations}  of (resolution of) singularities. This parameter space
will be the major tool
in our working definition as well (see \ref{ss:genan}).

However, even if one defines a certain  `genericity'  notion by eliminating a  discriminant
from a
parameter space (consisting of the pathological objects from the point of view of the discussion),
the next  hard major task is to exploit from the genericity  some key
geometric/numerical/cohomological properties.
E.g., in the present article this is done via Theorem B.

Regarding the problem to find the values of the analytic invariants associated with the
generic analytic type,
a crucial   obstruction was (before the present note)
the {\it lack of examples and experience}.
E.g.,
Laufer  in \cite{Laufer77} proved that
a generic elliptic singularity has geometric genus $p_g=1$,
but except this almost no other example  is known.
Even more, using the known statements of the literature, it is almost impossible to
 guess what are the possible topological candidates for the
invariants of the generic analytic structure. The expectation is that they should be certain
sharp topological bounds, but even if some topological bound is known, usually there are no tools to
prove its realization for the generic (or any) analytic structure.
The situation is exemplified rather trustworthily already by the geometric genus.
Wagreich already in 1970 in \cite{Wa70} defined topologically  the `arithmetical genus' $p_a$
of a normal surface singularity and for any non--rational germ (that is, when $p_g\not=0$) he proved that
$p_a\leq p_g$ (see \cite[p. 425]{Wa70}).
Though in some (easy) cases was known that they agree, analyzing the existing proofs
of the inequality
(see e.g. the very short proof in \cite{NO17}), one might think
that this inequality for germs with
complicated  topological types  probably is  extremely  week. However,
the point is that in the present note we prove that
(contrary to the first naive judgement)
the geometric analytic structure realizes
exactly this $p_a$.  For the other invariants (listed in Theorem A)
even the corresponding candidates were not on the table
(but we expect that they will have some relationship with lattice cohomology \cite{lattice}).

 In fact, even in this article we make the selection of a package of analytic invariants
 (organized around the cohomology of natural line bundles), for which we present the
 corresponding `package of topological  expressions', and we will treat, say, the
Hilbert--Samuel function/multiplicity/embedded-dimension package
 in a forthcoming manuscript (with rather different type of combinatorial answers).

\subsection{The working definition of
the generic analytic type}\label{ss:genan} Usually when we have a parameter space
for a family of geometric objects, the `generic object' might depend essentially
on the fact that  what kind of geometrical
problem we wish to solve, or, what kind of
anomalies we wish to avoid. Accordingly,  we determine a discriminant space
of the non--wished objects, and generic means  its complement.  In the present article all the
discrete  analytic invariants
we treat are basically guided by the cohomology groups of the natural line bundles (for their definition
see \cite{trieste}, \cite{OkumaRat} or \ref{ss:natline} here, they associate in a canonical way
a line bundle to any given Chern class). Hence, the discriminant spaces
(sitting in the base space of complete deformation spaces of Laufer  \cite{LaDef1})
are defined as the `jump loci' of the cohomology groups of the natural line bundles.
In section 3 we recall the needed results of Laufer regarding complete deformations of
some $\tX$, and we build on this our  working definition of general
analytic type.

Note that the natural line bundles are well--defined only if the
link is a rational homology sphere. Furthermore, this assumption appeared in the theory of
Abel maps as well.
Hence, in the article we also
impose this topological
restriction.

\section{Preliminaries and notations}\label{s:prel}
\subsection{Notations regarding a good resolution}\label{ss:2.1}
  \cite{Nfive,trieste,NCL,LPhd,NNI}
Let $(X,o)$ be the germ of a complex analytic normal surface singularity,
 and let us fix  a good resolution  $\phi:\widetilde{X}\to X$ of $(X,o)$.
Let $E$ be  the exceptional curve $\phi^{-1}(0)$ and  $\cup_{v\in\calv}E_v$ be
its irreducible decomposition. Define  $E_I:=\sum_{v\in I}E_v$ for any subset $I\subset \calv$.

We will assume that each $E_v$ is rational, and the dual graph is a tree. This happens exactly when the link
$M$ of $(X,o)$ is a rational homology sphere.

$L:=H_2(\widetilde{X},\mathbb{Z})$  is a lattice  endowed
with the natural  negative definite intersection form  $(\,,\,)$. It is
freely generated by the classes of  $\{E_v\}_{v\in\mathcal{V}}$.
 The dual lattice is $L'={\rm Hom}_\Z(L,\Z)=\{
l'\in L\otimes \Q\,:\, (l',L)\in\Z\}$. It  is generated
by the (anti)dual classes $\{E^*_v\}_{v\in\mathcal{V}}$ defined
by $(E^{*}_{v},E_{w})=-\delta_{vw}$ (where $\delta_{vw}$ stays for the  Kronecker symbol).
$L'$ is also  identified with $H^2(\tX,\Z)$.
The anticanonical cycle $Z_K\in L' $ is defined via the adjunction identities
$(Z_K,E_v)=E_v^2+2$ for all $v$.

All the $E_v$--coordinates of any $E^*_u$ are strict positive.
We define the (rational)  Lipman cone as $\calS':=\{l'\in L'\,:\, (l', E_v)\leq 0 \ \mbox{for all $v$}\}$.
As a monoid it is generated over $\bZ_{\geq 0}$ by $\{E^*_v\}_v$.

$L$ embeds into $L'$ with
 $ L'/L\simeq H_1(M,\mathbb{Z})$.  $L'/L$ is  abridged by $H$.
Each class $h\in H=L'/L$ has a unique representative $r_h\in L'$ in the semi-open cube
$\{\sum_vr_vE_v\in L'\,:\, r_v\in \bQ\cap [0,1)\}$, such that its class  $[r_h]$ is $h$.

There is a natural (partial) ordering of $L'$ and $L$: we write $l_1'\geq l_2'$ if
$l_1'-l_2'=\sum _v r_vE_v$ with all $r_v\geq 0$. We set $L_{\geq 0}=\{l\in L\,:\, l\geq 0\}$ and
$L_{>0}=L_{\geq 0}\setminus \{0\}$.

The support of a cycle $l=\sum n_vE_v$ is defined as  $|l|=\cup_{n_v\not=0}E_v$.

\subsection{The Abel map}\cite{NNI}\label{ss:Integ}
 Let $\pic(\tX)=H^1(\tX,\calO_{\tX}^*)$ be the group of isomorphic
classes of holomorphic line bundles on $\tX$.
The first Chern map $c_1:\pic(\tX)\to L'$ is surjective; write
$\pic^{l'}(\tX)=c_1^{-1}(l')$. Since $H^1(M,\Q)=0$, by the exponential
exact sequence on $\tX$ one has $\pic^0(\tX)\simeq H^1(\tX,\calO_{\tX})\simeq \C^{p_g}$,
where $p_g$ is the geometric genus.

Similarly, if $Z$ is an effective non--zero integral cycle supported by $E$, then ${\rm Pic}(Z)=H^1(Z,\calO_Z^*)$ denotes   the group of isomorphism classes
of invertible sheaves on $Z$. Again, it appears in the exact sequence $
0\to {\rm Pic}^0(Z)\to {\rm Pic}(Z)\stackrel{c_1} {\longrightarrow} L'(|Z|)\to 0$,
where ${\rm Pic}^0(Z)$ is identified with $H^1(Z,\calO_Z)$ by the exponential exact sequence.
Here  $L(|Z|)$ denotes the sublattice of $L$ generated by
the  base element $E_v\subset |Z|$, and $L'(|Z|)$ is its dual lattice.

For any $Z\in L_{>0}$ let
$\eca(Z)$  be the space of (analytic) effective Cartier divisors on 
$Z$. Their supports are zero--dimensional in $E$.
Taking the class of a Cartier divisor provides  the {\it Abel map}
$c:\eca(Z)\to \pic(Z)$.
Let
$\eca^{\tilde{l}}(Z)$ be the set of effective Cartier divisors with
Chern class $\tilde{l}\in L'(|Z|)$, i.e.
$\eca^{\tilde{l}}(Z):=c^{-1}(\pic^{\tilde{l}}(Z))$.
The restriction of $c$ is denoted by   $c^{\tilde{l}}(Z):\eca^{\tilde{l}}(Z)\to \pic^{\tilde{l}}(Z)$.

We also use the notation
$\eca^{l'}(Z):=\eca^{R(l')}(Z)$ and $\pic^{l'}(Z):= \pic^{R(l')}(Z)$ for any $l'\in L'$,
where $R:L'\to L'(|Z|)$  is the cohomological restriction,  dual to the
inclusion $L(|Z|)\hookrightarrow L$. (This means that
$R(E^*_v)=$the (anti)dual of $E_v$ in the lattice $L'(|Z|)$ if $E_v\subset |Z|$ and  $R(E^*_v)=0$ otherwise.)

A line bundle $\calL\in \pic^{\tilde{l}}(Z)$ is in the image
 ${\rm im}(c^{\tilde{l}})$ if and only if it has a section without fixed components, that is,
 if $H^0(Z,\calL)_{reg}\not=\emptyset $, where
$H^0(Z,\calL)_{reg}:=H^0(Z,\calL)\setminus \cup_v H^0(Z-E_v, \calL(-E_v))$.
Here the inclusion of $H^0(Z-E_v, \calL(-E_v))$ into $H^0(Z,\calL)$ is given by the long cohomological exact sequence associated with
$0\to \calL(-E_v)|_{Z-E_v}\to \calL\to \calL|_{E_v}\to 0$,
and it represents the subspace of sections, whose fixed components contain $E_v$.

By this definition (see (3.1.5) of \cite{NNI}) $\eca^{\tilde{l}}(Z)\not=\emptyset$ if and only if
$-\tilde{l}\in \calS'(|Z|)\setminus \{0\}$. It is advantageous to have a similar statement for
$\tilde{l}=0$ too, hence we redefine  $\eca^0(Z)$ as $\{\emptyset\}$, a set/space with one element
(the empty divisor), and $c^0:\eca ^0(Z)\to \pic^0(Z)$ by $c^0(\emptyset)=\calO_Z$. Then
\begin{equation}\label{eq:Chernzero}H^0(Z,\calL)_{reg}\not=\emptyset\ \Leftrightarrow\ \calL=\calO_Z\
\Leftrightarrow\ \calL\in {\rm im}(c^0)  \ \ \ \ (c_1(\calL)=0).\end{equation}
Then the `extended equivalence' reads as:
$\eca^{\tilde{l}}(Z)\not=\emptyset$ if and only if
$-\tilde{l}\in \calS'(|Z|)$. In such a case
$\eca^{\tilde{l}}(Z)$ is a smooth complex algebraic variety  of dimension $(\tilde{l},Z)$,
cf. \cite[Th. 3.1.10]{NNI}.
Furthermore, the Abel map  is an algebraic regular map.
It can be described using Laufer's duality as follows, cf.
 \cite{Laufer72}, \cite[p. 1281]{Laufer77} or \cite{NNI}. First, by Serre duality,
\begin{equation}\label{eq:LD}
H^1(\tX,\cO_{\tX})^*\simeq  H^1_c(\tX,\Omega^2_{\tX})\simeq
H^0(\tX\setminus E,\Omega^2_{\tX})/ H^0(\tX,\Omega^2_{\tX}).\end{equation}
An element of $H^0(\tX\setminus E,\Omega^2_{\tX})/ H^0(\tX,\Omega^2_{\tX})$ can be represented by
the class of a form $\tomega\in H^0(\tX\setminus E,\Omega^2_{\tX})$.
Furthermore,  an element $[\alpha]$ of
$H^1(\tX,\cO_{\tX})$ can be represented by a  $\check{C}ech$ {\it cocyle}
$\alpha_{ij}\in \cO(U_i\cap U_j)$, where $\{U_i\}_i$ is an open cover of $E$,
$U_i\cap U_j\cap U_k=\emptyset$, and each connected component of the
intersections $U_i\cap U_j$ is either a coordinate bidisc $B=\{|u|<2\epsilon,\ |v|<2\epsilon\}$
 with coordinates $(u,v)$, such that
$E\cap B\subset  \{ uv=0\}$,  or a punctured
coordinate bidisc $B=\{\epsilon/2<|v|<2\epsilon,\ |u|<2\epsilon\}$
with coordinates $(u,v)$, such that
$E\cap B=  \{ u=0\}$.
 Then,  Laufer's realization of the duality
 $H^0(\tX\setminus E,\Omega^2_{\tX})/ H^0(\tX,\Omega^2_{\tX})\otimes H^1(\tX,\cO_{\tX})\to \C$
  is
\begin{equation}\label{eq:Stokes}
\langle [\alpha],[\tomega]\rangle=
\sum_{B}\ \int_{|u|=\epsilon, \ |v|=\epsilon} \alpha_{ij}\tomega.
\end{equation}
In particular,
if $\tomega$ has no pole along $E$ in $B$, then the $B$--contribution in the above sum is zero.

This duality, via the isomorphism
$\exp: H^1(\tX,\cO_{\tX})\to c_1^{-1}(0)\subset H^1(\tX,\cO_{\tX}^*)=\pic(\tX)$, can be transported as follows, cf.
\cite{NNI}. (Here we present the case of  a peculiar  divisor
due to the fact that this version will be used later.)
Consider the following situation.
We fix a smooth point $p$ on $E$ ($p\in E_v$), a local bidisc $B\ni p$ with local coordinates $(u,v)$ such that $B\cap E=\{u=0\}$, $B=\{|u|,\, |v|<\epsilon\}$.
We assume that a certain form   $\tomega\in H^0(\tX\setminus E,\Omega^2_{\tX})$ has  local equation
$\tomega=\sum_{ i\in\Z,j\geq 0}a_{i,j} u^iv^jdu\wedge dv$ in $B$.
In the same time,
we fix a divisor $\widetilde{D}$ on $\tX$, whose local equation in $B$ is $v^\ell$,
 $\ell\geq 1$.
 Let $\widetilde{D}_{\underline{t}}$ be another divisor, which
is the same as $\widetilde{D}$ in the complement of $B$ and in $B$ its local equation is $(v+t+\sum_{k\geq 1, l\geq 0} t_{k,l}u^kv^l)^\ell$,
where all $t, \, t_{k,l}\in\C$ and  $|t|, \, |t_{k,l}|\ll 1$.
Then $\widetilde{D}_{\underline{t}}-\widetilde{D}$  is the divisor $\widetilde{D}'=
{\rm div}( g)$, where  $g:=((v+t+\sum_{k\geq 1, l\geq 0} t_{k,l}u^kv^l)/v)^\ell$,
supported in $B$. In particular, $\cO(\widetilde{D}')\in \pic^0(\tX)\subset H^1(\tX, \calO^*_{\tX})$ can be
represented by the cocycle $g|_{B^*}\in \cO^*(B^*)$, where $B^*=\{\epsilon/2<|v|<\epsilon, \, |u|<\epsilon\}$.  Therefore,
 $\log( g|_{B^*})$ is a cocycle in
$B^*$ representing its lifting into $H^1(\tX,\cO_{\tX})$. This  paired with $\tomega$ gives
for $\langle\langle \widetilde{D}_{\underline{t}},[\tomega]\rangle\rangle:=
\langle\exp^{-1} \cO_{\tX}(\widetilde{D}_t-\widetilde{D}),[\tomega]\rangle$ the
expression
\begin{equation}\label{eq:Tomega1}
\ell\cdot \int_{|u|=\epsilon, \ |v|=\epsilon} \log\Big(1+\frac{t+\sum_{k, l} t_{k,l}u^kv^l}{v}\Big)\cdot
 \sum_{ i\in\Z,j\geq 0}a_{i,j} u^iv^jdu\wedge dv.\end{equation}
If $\tomega$ has no pole then $\langle\langle \widetilde{D}_{\underline{t}},[\tomega]\rangle\rangle=0$.
As an example, assume that  $\tomega$ has the form $(h(u,v)/u^o)du\wedge dv$ with $h$ regular and $h(0,0)\not=0$, and $o\geq 1$, while
 $g=(v+tu^{o-1})/v$ and $\ell=1$, then
 \begin{equation}\label{eq:Tomega2}
\langle\langle \widetilde{D}_{\underline{t}},[\tomega]\rangle\rangle=
\int_{|u|=\epsilon, \ |v|=\epsilon} \log\Big(1+\frac{tu^{o-1}}{v}\Big)\cdot
 \frac{h}{u^o}\,du\wedge dv= c\cdot t +\{\mbox{higher order terms}\} \ \  (c\in \C^*). \end{equation}
If $Z\gg0$ then   $H^0(\tX\setminus E,\Omega^2_{\tX})/ H^0(\tX,\Omega^2_{\tX})\simeq
 H^0(\tX,\Omega^2_{\tX}(Z))/ H^0(\tX,\Omega^2_{\tX})$. Furthermore,
if $\tomega_1,\ldots, \tomega_{p_g}$ are representatives of  a basis of  this vector space
and $\widetilde{D}_{\underline{t}}$ is considered as a path in $\eca^{-\ell E^*_v}(Z)$, then
$\widetilde{D}_{\underline{t}}\mapsto (\langle\langle \widetilde{D}_{\underline{t}},
[\tomega_1]\rangle\rangle, \ldots,
\langle\langle \widetilde{D}_{\underline{t}},[\tomega_{p_g}]\rangle\rangle)$
is the restriction of the Abel map to $\widetilde{D}_{\underline{t}}$ (associated with $Z$,
and shifted by the image of $\widetilde{D}$) (cf. \cite{NNI}).

\section{Resolutions  with generic analytic structure}\label{s:GASt}

\subsection{The setup}\label{ss:GAStsetup} We fix a topological type of a normal surface singularity.
This means that we fix  either the $C^\infty$ oriented diffeomorphism type of the link,
or, equivalently,
one of the dual  graphs of a good resolution (all of them are equivalent up to blowing up/down
rational $(-1)$--vertices). We assume that the link is a rational homology sphere, that is, the graph is a tree of rational vertices.

Any such topological type might support several analytic structures. The moduli
space of the possible analytic structures is not described yet in the literature, hence we cannot
rely on it. In particular, the `generic analytic structure', as a `generic' point of this moduli space,
 in this way is not well--defined. However, in order to run/prove the concrete properties
 regarding generic analytic structures,
 instead of such theoretical definition it would be even much better
  to consider a definition based
 on  a list of stability properties under certain concrete deformations
 (whose validity could be expected for the `generic'  analytic structure in the presence of a
 classification space).
 Hence, for us in this note, a generic analytic structure will be
 a structure, which will satisfy such stability properties.
 In order to define them it is convenient to fix a resolution graph $\Gamma$ and treat
 deformation of analytic structures supported on
 resolution spaces having dual graph $\Gamma$.

The type of stability we wish to have is the following.
The topological type (or, the graph $\Gamma$) determines a lower bound
 for the possible values of the geometric genus (which usually depends on the analytic type).
 Let MIN$(\Gamma)$ be the unique optimal bound, that is, MIN$(\Gamma)\leq p_g(X,o)$ for any
 singularity $(X,o)$ which admits  $\Gamma$ as a resolution graph,
 and MIN$(\Gamma)=p_g(X,o)$ for some $(X,o)$.
 Then one of the requirements for the `generic analytic structure' $(X_{gen},o)$ is that
$p_g(X_{gen},o)={\rm MIN}(\Gamma)$.
(In the body of the paper MIN$(\Gamma)$ will be determined explicitly.)
However, we will need several similar stability
requirements involving other line bundles as well (besides the trivial one, which provides $p_g$).
For their definition we need a preparation.
\subsection{Laufer's results}\label{ss:LauferDef}
In this subsection we review some results of Laufer regarding deformations of the analytic structure on a resolution space of a normal surface singularity with fixed resolution graph
(and deformations of non--reduced analytic spaces supported on  exceptional curves) \cite{LaDef1}.

First, let us fix a normal surface singularity $(X,o)$ and a good resolution $\phi:(\tX,E)\to (X,o)$
with reduced exceptional curve $E=\phi^{-1}(o)$, whose irreducible decomposition is $\cup_{v\in\calv}E_v$ and dual graph $\Gamma$.
Let $\cali_v$ be the ideal sheaf of $E_v\subset \tX$. Then for arbitrary
 positive integers $\{r_v\}_{v\in \calv}$ one defines two objects, an analytic one and a topological (combinatorial) one.
 At analytic level, one sets  the ideal sheaf  $\cali(r):=\prod_v \cali_v^{r_v}$
 and the non--reduces space $Z(r)$ with structure sheaf
 $\calO_{Z(r)}:=\calO_{\tX}/\cali(r)$ supported on $E$.

  The topological object is a graph decorated with multiplicities, denoted by $\Gamma(r)$. As a non--decorated graph $\Gamma(r)$ coincides with the graph $\Gamma$ without decorations. Additionally each vertex $v$ has a
  `multiplicity decoration' $r_v$, and we put also the self--intersection  decoration $E_v^2$
  whenever $r_v>1$. (Hence, the vertex $v$ does not inherit the self--intersection decoration
  of $v$ if $r_v=1$).
 Note that the  abstract 1--dimensional analytic space $Z(r)$ determines by its reduced structure
 the shape of the dual graph $\Gamma$, and by its non--reduced structure
 all the multiplicities $\{r_v\}_{v\in\calv}$, and additionally,
 all the self--intersection numbers $E_v^2$ for those $v$'s when  $r_v>1$
 (see \cite[Lemma 3.1]{LaDef1}).

 We say that the space $Z(r)$ has topological type $\Gamma(r)$.

Clearly, the analytic structure of $(X,o)$, hence of $\tX$ too, determines each 1--dimensional
non--reduced space $Z(r)$.  The converse is also true in the following sense.
\begin{theorem}\label{th:La1} \ \cite[Th. 6.20]{Lauferbook},\cite[Prop. 3.8]{LaDef1}
(a) Consider an abstract 1--dimensional space $Z(r)$, whose topological type
$\Gamma(r)$ can be completed to a negative definite graph $\Gamma$ (or, lattice $L$).
Then there exists a 2--dimensional
manifold $\tX$ in which $Z(r)$ can be embedded with support $E$
such that the intersection matrix inherited from the embedding $E\subset \tX$ is the
negative definite lattice $L$.
In particular (since by Grauert theorem \cite{GRa}
the exceptional locus $E$ in $\tX$ can be contracted to a normal singularity),
any such $Z(r)$ is always associated with a normal surface singularity (as above).

(b)  Suppose that we have two singularities $(X,o)$ and $(X',o)$ with good resolutions as above with the
same resolution graph $\Gamma$. Depending solely on $\Gamma$,
 the integers $\{r_v\}_v$ may be chosen so
large that if $\calO_{Z(r)}\simeq \calO_{Z'(r)}$, then $E\subset \tX$ and $E'\subset \tX'$ have
biholomorphically equivalent neighbourhoods via a map taking $E$ to $E'$.
(For a concrete estimate how large $r$ should be see Theorem 6.20 in \cite{Lauferbook}.)
\end{theorem}
In particular, in the deformation theory of $\tX$ it is enough to consider the deformations of
non--reduced spaces of type $Z(r)$.

Fix a non--reduced 1--dimensional space $Z=Z(r)$ with topological type $\Gamma(r)$.
Following Laufer and for technical reasons (partly motivated by further applications in the
forthcoming continuations of the series of manuscripts)
we also choose a closed subspace $Y$ of $Z$ (whose support can be smaller, it can be even empty).
More precisely, $(Z,Y)$ locally is isomorphic with $(\C\{x,y\}/(x^ay^b),\C\{x,y\}/(x^cy^d))$,
where $a\geq c\geq 0$, $b\geq d\geq 0$, $a>0$.
  The ideal of $Y$ in $\calO_Z$ is denoted by $\cali_Y$.
\begin{definition}\label{def:1}\ \cite[Def. 2.1]{LaDef1}
 A deformation of $Z$, fixing $Y$, consists of the following data:

(i) There  exists an analytic space $\calz$ and a proper map $\omegl:\calz\to Q$, where
$Q$ is a manifold containing a distinguished point $0$.

(ii) Over a point $q\in Q$ the fiber $Z_q$ is the subspace of $\calz$ determined by the ideal
sheaf $\omegl^* (\mathfrak{m}_q)$ (where $\mathfrak{m}_q$ is the maximal ideal of $q$). $Z$
is isomorphic with $Z_0$, usually they are  identified.

(iii) $\omegl$ is a trivial deformation of $Y$ (that is, there is a closed subspace
$\caly\subset \calz$ and the restriction of $\omegl$ to $\caly$ is a trivial deformation of $Y$).

(iv) $\omegl$ is {\it locally trivial} in a way which extends the trivial deformation $\omegl|_{\caly}$.
This means  that for ant $q\in Q$ and $z\in \calz$ there exist a neighborhood $W$ of $z$ in $\calz$,
a neighborhood $V$ of $z$ in $Z_q$, a neighborhood $U$ of $q$ in $Q$, and an isomorphism
$\phi:W\to V\times U$ such that $\omegl|_W=pr_2\circ \phi$ (compatibly with the trivialization
of $\caly$ from (iii)), where $pr_2$ is the second projection; for more see  [loc.cit.].
\end{definition}
One verifies that under deformations (with connected base space) the topological type of the fibers
$Z_q$,  namely $\Gamma(r)$, stays constant (see \cite[Lemma 3.1]{LaDef1}).

\begin{definition}\label{def:2}\ \cite[Def. 2.4]{LaDef1}
A deformation $\omegl:\calz\to Q$ of $Z$, fixing $Y$, is complete at $0$ if, given any deformation
$\tau:{\mathcal P}\to R$ of $Z$ fixing $Y$, there is a neighbourhood $R'$ of $0$ in $R$  and a
 holomorphic map $f:R'\to Q$ such that $\tau$ restricted to $\tau^{-1}(R')$ is the deformation
$f^*\omegl$. Furthermore, $\omegl$ is complete if it is complete at each point $q\in Q$.
\end{definition}
Laufer proved the following  results.
\begin{theorem}\label{th:La2}\ \cite[Theorems 2.1, 2.3, 3.4, 3.6]{LaDef1}
Let $\theta_{Z,Y}={\mathcal Hom}_{Z}(\Omega^1_Z,\cali_Y)$ be the sheaf of germs of vector fields on $Z$,  which vanish on $Y$, and let $\omegl :\calz\to Q$ be a deformation of $Z$, fixing $Y$.

 (a) If the Kodaira--Spencer map
 $\rho_0:T_0Q\to H^1(Z,\theta_{Z,Y})$ is surjective then $\omegl$ is complete at $0$.

 (b) If $\rho_0$ is surjective then $\rho_q$ is surjective for all $q$ sufficiently near to $0$.

 (c)  There exists a deformation $\omegl$ with $\rho_0$ bijective. In such a case in a neighbourhood $U$ of $0$ the deformation is essentially unique, and  the fiber above $q$ is isomorphic to $Z$
 for only at most countably many $q$ in $U$.
\end{theorem}

\bekezdes\label{bek:funk} {\bf Functoriality.} Let $Z'$ be a closed subspace of $Z$ such that
$\cali_{Z'}\subset \cali_Y\subset \calO_Z$. Then there is a natural reduction of pairs
$(\calO_Z,\calO_Y)\to (\calO_{Z'},\calO_Y)$. Hence, any deformation $\omegl:\calz\to Q$
of $Z$ fixing $Y$ reduces to a deformation $\omegl':\calz'\to Q$
of $Z'$ fixing $Y$. Furthermore, if $\omegl$ is complete then $\omegl'$ is automatically
complete as well (since $H^1(Z,\theta_{Z,Y})\to H^1(Z',\theta_{Z',Y})$ is onto).

\subsection{The `0--generic analytic structure'}\label{ss:0gen}
We wish to define when is the
 analytic structure of a  fiber  $Z_q$ $(q\in Q)$  of a deformation  `generic'.
 We proceed in two steps. The `0--genericity' is the first one (corresponding to the Chern
 class $l'=0$), which will be defined in this subsection.

 It is rather advantageous to set a definition, which is compatible with respect to all
  the restrictions $\calO_{Z}\to\calO_{Z'}$.
  In order to do this, let us fix
  the coefficients $\tir=\{\tir_v\}_v$ so large that for them Theorem \ref{th:La1} is valid.
  In this way basically we fix a resolution $(\tX,E)$ and some large infinitesimal
  neighbourhood $Z(\tir)$ associated with it. 
Moreover,  let us also fix a {\it complete}
 deformation $\omegl(\tir):\calz(\tir)\to Q$  whose fibers have the topological  type
 of $\Gamma(\tir)$.  Next, we consider all the other coefficient sets $r:=\{r_v\}_v$ such that
 $0\leq r_v\leq \tir_v$ for all $v$, not all $r_v=0$. Such a  choice,
 by restriction as in \ref{bek:funk}, automatically provides a
 deformation $\omegl(r):\calz(r)\to Q$.
Then set
\begin{equation}\label{eq:Delta0}
\Delta(0,r):=\{q\in Q\,:\,  \mbox{$h^i(Z(r)_q, \calO_{Z(r)_q})$ is not constant in a neighbourhood of
$q$ for some $i$}\}.
\end{equation}

Then $\Delta(0,r)$ is a  closed (reduced) proper subspace of $Q$, see \cite{Ri74,Ri76}
 (one can use also an argument similar to Lemma \ref{lem:semic} written for  $l'=0$). Define
$\Delta^0(\tir):= \cup_{ r_v\leq \tir_v}\Delta(0, r)$. Then
 $\Delta ^0(\tir)$ is also closed  and $\Delta^0(\tir)\not=Q$.
 \begin{definition}\label{def:GEN0}
 We say that the fiber $Z(\tir)_q$ of $\omegl(\tir):\calz(\tir)\to Q$ is 0--generic if
$q\in Q\setminus \Delta^0(\tir)$.
 \end{definition}
 Next, we wish to generalize this definition for all Chern classes $l'\in L' $, or, for all
 `natural line bundles', as generalizations of the trivial bundle corresponding to $l'=0$.
\subsection{Natural line bundles}\label{ss:natline}
Let us start  again with a good resolution $\phi:(\tX,E)\to (X,o)$
of a  normal surface singularity with rational homology sphere link,
and consider the cohomology exact sequence associated with the exponential exact sequence of sheaves
\begin{equation}\label{eq:exp}
0\to {\rm Pic}^0(\tX)\stackrel{\epsilon}{\longrightarrow}
{\rm Pic}(\tX)\stackrel{c_1}{\longrightarrow}H^2(\tX,\Z)\to 0.
\end{equation}
Here $c_1(\calL)\in H^2(\tX,\Z)=L' $ is the first Chern class of $\calL$.
Then, see e.g. \cite{OkumaRat,trieste},
there exists a unique homomorphism (split)
 $s:L'\to {\rm Pic}(\tX)$  of $c_1$   such that $c_1\circ s=id$ and
$s$ restricted to $L$ is $l\mapsto \calO_{\tX}(l)$.
The line bundles $s(l')$ are called {\it natural line
bundles } of $\tX$, and are denoted by $\calO_{\tX}(l')$. For several definitions of
them see  \cite{trieste}.
E.g., $\calL$ is natural if and only if one of its power has the form $\calO_{\tX}(l)$
for some {\it integral} cycle $l\in L$ supported on $E$.
Here  we recall another construction from \cite{OkumaRat,trieste},
which will be extended later to the deformations space of singularities.

Fix some $l'\in L'$ and let $n$ be the order of its class in $L'/L$.
Then $nl'$ is an integral cycle; its reinterpretation as a
 divisor supported on $E$ will be denoted by   ${\rm div}(nl')$.
We claim that there exists a divisor $D=D(l')$ in $\tX$ such that one has a linear equivalence
$nD\sim {\rm div}(nl')$ and $c_1(\calO_{\tX}(D))=l'$. Furthermore,
$D(l') $ is unique up to linear equivalence, hence
$l'\mapsto \calO_{\tX}(D(l'))$ is the wished split of (\ref{eq:exp}).
Indeed, since $c_1$ is onto, there exists a divisor $D_1$ such that $c_1(\calO_{\tX}(D_1))=l'$.
Hence $\calO_{\tX}(nD_1-{\rm div}(nl'))$ has the form $\epsilon(\calL)$ for some
$\calL\in {\rm Pic}^0(\tX)=H^1(\tX,\calO_{\tX})=\C^{p_g}$.
Define $D_2$ such that $\calO_{\tX}(D_2)=\frac{1}{n}\calL$ in
$H^1(\tX,\calO_{\tX})$. Then $D_1-D_2$ works. The uniqueness  follows from the fact that
${\rm Pic}^0(\tX)$ is torsion free.

The following warning is appropriate.
Note that if $\tX_1$ is a connected small convenient  neighbourhood
of the union of some of the exceptional divisors (hence $\tX_1$ also stays as the resolution
of the singularity obtained by contraction of that union of exceptional  curves) then one can repeat the definition of
natural line bundles at the level of $\tX_1$ as well. However,  the restriction to
$\tX_1$ of a natural line bundle of $\tX$ (even of type
$\calO_{\tX}(l)$ with $l$ integral cycle supported on $E$)  usually is not natural on $\tX_1$:
$\calO_{\tX}(l')|_{\tX_1}\not= \calO_{\tX_1}(R(l'))$
 (where $R:H^2(\tX,\Z)\to H^2(\tX_1,\Z)$ is the natural restriction), though their Chern classes coincide.

In the sequel we will deal with the family of `restricted natural line bundles'
 obtained by {\it restrictions
of $\calO_{\tX}(l')$}. Even if we need to descend to a `lower level' $\tX_1$ with smaller
exceptional curve, or to any cycle $Z$ with support included in $E$ (but not necessarily $E$)
our `restricted natural line bundles' will be associated with Chern classes  $l'\in L'=L'(\tX)$
via the restrictions $\pic(\tX)\to \pic(\tX_1)$ or $\pic(\tX)\to \pic(Z)$ of bundles of type
$\calO_{\tX}(l') \in {\rm Pic}(\tX)$.
This basically means that we fix a tower of singularities $\{\tX_1\}_{\tX_1\subset \tX}$,
or $\{\calO_Z\}_{|Z|\subset E}$, determined by the `top level' $\tX$, and all the restricted
natural line bundles,
even at intermediate levels, are restrictions  from the top level.

We use the notations $\calO_{\tX_1}(l'):=\calO_{\tX}(l')|_{\tX_1}$ and $\calO_Z(l'):=\calO_{\tX}(l')|_Z$ respectively.

\subsection{The universal family of natural line bundles}\label{bek:3.3.4}
Next, we wish to extend the definition of the line bundles
$\calO_{Z}(l')$ to the total space of a deformation (at least locally,  over small balls
in the complement of $\Delta^0(\tir)$).

We fix some $Z=Z(\tir)$ with all $\tir_v\gg0$, supported on $E$, such that Theorem
\ref{th:La1} is valid (similarly as in \ref{ss:0gen}). Fix also some $Y\subset Z$,
 and a complete deformation $\omegl:\calz(\tir)\to Q$ of $(Z,Y)$ as in Definition \ref{def:1}
such that all the fibers have the same fixed topological type $\Gamma(\tir)$.
We consider the discriminant $\Delta^0(\tir) \subset Q$, and we fix some
$q_0\in Q\setminus \Delta^0(\tir)$,
and a small ball $U$,  $q_0\in U\subset Q\setminus \Delta^0(\tir)$.
Above $U$ the topologically trivial family of irreducible
exceptional curves
form the irreducible divisors $\{{\cale}_v\}_v$, such that ${\cale}_v$ above any point $q\in U$ is the
corresponding irreducible exceptional curve $E_{v,q}$
of $\tX_q$. With the notations of the previous paragraph, if $nl'$
has the form
$\sum_vn_vE_v$ write ${\rm div}_{\omegl}(nl'):=\sum _vn_v{\cale}_v$ for the corresponding divisor
in $\omegl^{-1}(U)$. Since $U$ is contractible, one has $H^2(\omegl^{-1}(U),\Z)=L'$
and $H^1(\omegl^{-1}(U),\Z)=0$,
hence  the exponential exact sequence on $\omegl^{-1}(U)$ gives
  \begin{equation}\label{eq:exp2}
0\to {\rm Pic}^0(\omegl^{-1}(U))\longrightarrow
{\rm Pic}(\omegl^{-1}(U))\stackrel{c_1}{\longrightarrow}L' \to H^2(\omegl^{-1}(U), \calO_{\omegl^{-1}(U)}).
\end{equation}
 \begin{lemma}\label{lem:LerayPic} $H^2(\omegl^{-1}(U), \calO_{\omegl^{-1}(U)})=0$ and
the first Chern class  morphism $c_1$ in (\ref{eq:exp2})
 is onto.
 \end{lemma}
 \begin{proof}
 We use the Leray spectral sequence.
 Recall, see e.g. EGA III.2 \S 7, or \cite{osserman},
that if $q\mapsto h^i(Z(\tir)_q, \calO_{Z(\tir)_q})$ is constant over
some open set $U$ (and all $i$) then
$R^i \omegl(\tir)_*\calO_{\calz(\tir)}$ is locally free over $U$ and
$R^i \omegl(\tir)_*\calO_{\calz(\tir)}\otimes _{\calO_U}\C(q)\to H^i(Z(\tir)_q, \calO_{Z(\tir)_q})$
is an isomorphism for $q\in U$.

Hence, since $R^i\omegl_*\calO_{\omegl^{-1}(U)}$ is locally free,
$H^i(U, R^{2-i}\omegl_*\calO_{\omegl^{-1}(U)})=0$ for $i>0$.
 On the other hand, $R^2\omegl_*\calO_{\omegl^{-1}(U)}=0$ since
 $R^2\omegl_*\calO_{\omegl^{-1}(U)}\otimes _{\calO_{U}}\C(q)\to H^2(Z(\tir)_q, \calO_{Z(\tir)_q})$ is an
 isomorphism and  $H^2(Z(\tir)_q, \calO_{Z(\tir)_q})=0$  by dimension argument.
 \end{proof}

Then,  if in the above  construction of the split of $c_1$ in (\ref{eq:exp})
we replace $\tX$  by $\omegl^{-1}(U)$ and ${\rm  div }(nl')$ by
${\rm div}_{\omegl}(nl')$,  we get the following statement.
\begin{lemma}\label{lem:natline}
For any $l'\in L'$
there exists a divisor $D_\omegl(l')$ in $\omegl^{-1}(U)$ such that one has a linear equivalence
$nD_\omegl(l')\sim {\rm div}_\omegl(nl')$ in $\omegl^{-1}(U)$
 and $c_1(\calO_{\omegl^{-1}(U)}(D_\omegl(l'))=l'$. Furthermore, $D_\omegl(l') $ is unique up to linear equivalence, hence $l'\mapsto \calO_{\omegl^{-1}(U)}(D_\omegl(l'))$ is a split of (\ref{eq:exp2})
 which extends the natural split $L\ni \sum_vm_vE_v\mapsto \calO_{\omegl^{-1}(U)}(\sum_v
 m_v{\cale}_v)$ over $L$. Since  ${\rm Pic}^0(\omegl^{-1}(U))=H^1 (\omegl^{-1}(U),
 \calO_{\omegl^{-1}(U)})$ is torsion free,
 there exists a unique split over $L'$ with this extension property.
\end{lemma}
Let us summarize what we obtained: For any  $q_0\in  Q\setminus \Delta^0(\tir)$, and
small ball $U$ with  $q_0\in U\subset Q\setminus \Delta^0(\tir)$,
 we have defined for each $l'\in L'$ a line bundle $\calO_{\omegl^{-1}(U)}(D_\omegl(l'))$
in ${\rm Pic}(\omegl^{-1}(U))$, such that its restriction  to each fiber $Z(\tir)_q$  is the line bundle
$\calO_{Z(\tir)_q}(l')$. Let us denote it by
  $\calO_{\omegl^{-1}(U)}(l')$.

\subsection{The semicontinuity of $q\mapsto h^1(Z_q,\calO_{Z_q}(l'))$} \label{ss:semic}
We fix a complete deformation $\lambda: \calz(\tir)\to Q$,
and we consider the
set of multiplicities $r_v\leq \tir_v$, not all zero, as in \ref{ss:0gen}.
Then, for each $r$,  we have a restricted
deformation $\omegl(r):\calz(r)\to Q$ of $Z(r)$  as in \ref{bek:3.3.4}.
\begin{lemma}\label{lem:semic} For any restricted natural line bundle the map
$q\mapsto h^i(Z(r)_q,\calO_{Z(r)_q}(l'))$ is semicontinuous over $Q\setminus \Delta^0(\tir)$,
for  $i=0,1$.
\end{lemma}
\noindent
(Note that if each $r_v>1$ then the
 intersection form on $\Gamma(r)$  is well--defined. In particular,
 the semicontinuities of $h^0$ and $h^1$ are equivalent,
since $h^0-h^1=(Z(r),l')+\chi(Z(r))$ by Riemann--Roch.)
 \begin{proof}
 We fix a small ball $U$ in  $Q\setminus \Delta^0(\tir)$ as in subsection \ref{bek:3.3.4},
 and we run $q\in U$.

 Let us denote (as above)
 the exceptional curves in the fiber $\omegl(r)^{-1}(q)$ by $\{E_{v,q}\}_v$,
 hence the cycle  $Z(r)_q$ is $\sum_vr_vE_{v,q}$.
 Then one has the short exact sequence of sheaves
 $$0\to \calO_{Z(r)_q}\otimes \calO_{\omegl^{-1}(U)}(l')\to \oplus_v \calO_{r_vE_{v,q}}\otimes \calO_{\omegl^{-1}(U)}(l')
 \to \oplus_{(v,w)} \C\{x,y\}/(x^{r_v}y^{r_w})\to 0,$$
 where the sum in the last term runs over the edges $(v,w)$ of $\Gamma(r)$.
 This gives  the Mayer--Vietoris exact sequence
 \begin{equation*}\label{eq:MV}
0\to
H^0(Z(r)_q,\calO_{\omegl^{-1}(U)}(l')|_{Z(r)_q}) \to
\oplus_v H^0(r_vE_{v,q},\calO_{\omegl^{-1}(U)}(l')|_{r_vE_{v,q}})
\stackrel{\delta}{\longrightarrow}
  \oplus_{(v,w)} \C\{x,y\}/(x^{r_v}y^{r_w}) \to \ldots
  \end{equation*}
Next, we analyse the vector space $H^0(r_vE_{v,q},\calO_{\omegl^{-1}(U)}(l')|_{r_vE_{v,q}})$ for any  $v$. Let us fix an arbitrary $q_0\in U$.
Note that a singularity with a resolution consisting only one rational irreducible
divisor is taut, see \cite{Laufer73}, hence the analytic family $\{Z(\tir)_q\}_q$
restricted to $\{r_vE_{v,q}\}_v$
over a small neighbourhood $U'\subset U $ of
$q_0$ can be trivialized. Furthermore, ${\rm Pic }^0(r_vE_{v,q})=0$, hence the line bundle
$\calO_{\omegl^{-1}(U)}(l')|_{r_vE_{v,q}}$ is uniquely determined topologically by $l'$ and $r$.
Hence, $\calO_{\omegl^{-1}(U)}(l')|_{r_vE_{v,q}}$ also can be trivialised over a small $U'$.
In particular, by these trivializations,
$H^0(r_vE_{v,q},\calO_{\omegl^{-1}(U)}(l')|_{r_vE_{v,q}})$ can be replaced by the fixed
  $H^0(r_vE_{v,q_0},\calO_{\omegl^{-1}(U)}(l')|_{r_vE_{v,q_0}})$, and the $q$--dependence is codified in the
restriction   morphism $\delta$. Hence, there exists a morphism
   \begin{equation}\label{eq:MV2}
 \oplus_v H^0(r_vE_{v,q_0},\calO_{\omegl^{-1}(U)}(l')|_{r_vE_{v,q_0}})
\stackrel{\delta(q)}{\longrightarrow}
  \oplus_{(v,w)} \C\{x,y\}/(x^{r_v}y^{r_w})
  \end{equation}
  whose kernel is
$H^0(Z(r)_q,\calO_{Z(r)_q}(l'))$.
Since the rank of $\delta(q)$ is semicontinuous, the statement follows for $h^0$.
But   $h^1(Z(r)_q,\calO_{Z(r)_q}(l'))=\dim {\rm coker}(\delta(q))+
h^1(r_vE_{v,q},\calO_{\omegl^{-1}(U)}(l')|_{r_vE_{v,q}})$, and the second term in this last sum  is also
topological and constant (by the same argument as above), hence semicontinuity for $h^1$ follows as well.
 \end{proof}

 \subsection{The `generic analytic structure'}\label{ss:gen}
 Now we are ready to give the definition of the   `generic structure'.   Let us  fix a {\it complete}
 deformation $\omegl(\tir):\calz(\tir)\to Q$  as in \ref{ss:0gen} (with $\tir_v$ large)
 whose fibers have the topological  type
 of $\Gamma(\tir)$.  Similarly as there,
  we consider all the other coefficient sets $r:=\{r_v\}_v$ such that
 $r_v\leq \tir_v$ for all $v$, not all zero, and the induced
 deformations $\omegl(r):\calz(r)\to Q$.
Then  for any $l'\in L'$ consider
 \begin{equation}\label{eq:MIN}
 {\rm MIN}(l',r):= \min_{q\in Q\setminus \Delta^0(\tir)}\{h^1(Z(r)_q,
 \calO_{Z(r)_q}(l'))\}\end{equation}
 and
 \begin{equation}\label{eq:Delta}
  \Delta(l',r):= \mbox{closure of}\ \{q\in Q\setminus \Delta^0(\tir)
  \,:\, h^1(Z(r)_q,\calO_{Z(r)_q}(l'))> {\rm MIN}(l',r)\}.
  \end{equation}
 Then $\Delta(l',r)$ is a  closed (reduced) proper subspace of $Q$
 (for this use e.g.
 an argument as in the proof of Lemma \ref{lem:semic}, or  \cite{Ri74,Ri76}).
Then set the countable union of closed  proper subspaces
$\Delta(\tir):=(\cup_{l'\in L'}\ \cup_{r_v\leq \tir_v}\Delta(l',r))\cup \Delta^0(\tir)$.
Clearly, $\Delta (\tir)\subsetneqq Q$.
\begin{definition}\label{def:GEN} (a)
For a fixed $\Gamma(\tir)$ and for any complete deformation
$\omegl(\tir):\calz(\tir)\to Q$ (with all $\tir_v\gg 0$)
 we say that the fiber $Z(\tir)_q$ of $\omegl(\tir):\calz(\tir)\to Q$ is generic if
$q\in Q\setminus \Delta(\tir)$.

(b) Consider a singularity $(X,o)$  and one of its resolutions
$\tX$ with dual graph $\Gamma$. We say that the analytic type on
$\tX$ is generic if there exists $\tir\gg 0$,
 and a complete deformation $\omegl(\tir):\calz(\tir)\to Q$ with fibers of topological type $\Gamma(\tir)$,
and   $q\in Q\setminus \Delta(\tir)$ such that  $\omegl(\tir)^{-1}(q)=\calO_{\tX}|_{ \sum_v \tir_vE_v}$.

\end{definition}
\begin{remark}\label{rem:generic}
(a) Fix any 1--dimensional space $Z$ with fixed topology $\Gamma(\tir)$
with all $\tir_v\gg 0$.
Then in any complete deformation $\omegl$ of $Z$ there exists a generic structure arbitrary close
to $Z$.

(b) Though the above construction  does not automatically imply that
$Q\setminus \Delta(\tir)$ is open, for any $q_0\in Q\setminus \Delta(\tir)$ and for any
{\it finite } set $FL'\subset L'$ there exists a small neighbourhood $U$ of $q_0$ such  that
$ h^1(\calO_{Z(r)_q},\calO_{Z(r)_q}(l'))= {\rm MIN}(l',r)$ for any
$r$ (as above),  $l'\in FL'$, and $q\in U$.

(c) 
Fix  a complete deformation $\omegl:\calz(\tir)\to Q$ of some $(Z,Y)$ with some fixed
 $\tir_v\gg0$ as above.
Then, by Theorem \ref{th:La1}{\it (b)} for any $q\in Q$ the fiber $Z(\tir)_q$
determines uniquely a holomorphic
neighborhood  $\tX_q$ of   $E$.  (Some  $\{\tir_v\}_v$ very large
works uniformly for all fibers, since a convenient $\{\tir_v\}_v$ can be chosen topologically.)
Furthermore,  $h^1(\tX_q,\calO_{\tX_q})$ can be
recovered from $\omegl$
as  $h^1(Z(\tir)_q,\calO_{Z(\tir)_q})$ by the formal function theorem.
This is the geometric genus of the singularity $(X_q,o)$ obtained by contracting $E$ in this $\tX_q$.
Since  $\Delta(0,\tir)=
\{q\in Q\,:\, p_g(X_q,o)={\rm MIN}(\Gamma)\}$ is part of the discriminant $\Delta(\tir)$ (and it is closed),
for any `generic' $q\in Q\setminus \Delta (\tir)$ there is a ball
$q\in U\subset Q\setminus \Delta (0,\tir)$
such that $\omegl$ simultaneously blows down to a flat family
${\mathcal X}\to U$. This follows from \cite{Ri74,Ri76,Wa76} by the constancy of $\Gamma$ and $p_g$.
\end{remark}
\subsection{Extension of  sections.} \label{ss:extsec}
 Consider a complete deformation $\lambda(\tir):{\mathcal Z}(\tir)\to Q$
as above, and {\it let $Z(\tir)_q$  be a generic fiber} as in Definition \ref{def:GEN}. Let
$U$ be a small neighbourhood of $q$ such that $U\subset Q\setminus \Delta^0(\tir)$.
For any $l'\in L'$ fixed consider the universal family of line bundles
$\calO_{\lambda^{-1}(U)}(D_\lambda (l'))$ constructed in subsection \ref{bek:3.3.4}.
Fix also some  $r:=\{r_v\}_v$ ($0\leq r_v\leq \tir_v$ for all $v$, not all $r_v=0$, as above).
Assume that $\calO_{Z(r)_q}(l')=\calO_{\lambda^{-1}(U)}(D_\lambda (l'))|_{Z(r)_q}$
admits a global section $s\in H^0(Z(r)_q,\calO_{Z(r)_q}(l'))$  without fixed components.
\begin{lemma}\label{lem:extension}\ After decreasing $U$ if it necessary, the following facts hold:

(a) the section $s$ has an extension
${\mathfrak s}\in H^0(\lambda(r)^{-1}(U),\calO_{\lambda(r)^{-1}(U)}(D_\lambda (l'))$ with
${\mathfrak s}_q=s$.

(b) ${\mathfrak s}_{q'}$ ($q'\in U,\  q'\not=q$) has no fixed components either.
\end{lemma}
\begin{proof} (a)
Since $Z(\tir)_q$ is generic, $q$  does not
 sit in the union of the discriminant spaces considered in \ref{ss:gen}. In that subsection we considered
all the discriminants associated with all the Chern classes and the `$r$--tower', hence, in particular,
we had countably many discriminant obstructions. By assumption, $q$ is not contained in any of these.
In this proof we have to concentrate on the Chern class $l'$ and the tower level $Z(r)$, hence only one discriminant.  In particular,
$q\in Q$ has a small neighbourhood which does not intersect it.
Therefore, decreasing the representative of  $(Q,q)$ we get the stability of the corresponding $h^1$--cohomology sheaves. Furthermore, $\omegl$ is proper,
 $\calO_{\lambda(r)^{-1}(U)}(D_\lambda (l'))$ is
coherent, and $q'\mapsto h^1(Z(r)_{q'},\calO_{Z(r)_{q'}}(l'))$ is constant.
 Hence by EGA III.2 \S 7
(or, see e.g. \cite{osserman}),
$R^0\omegl_*(\calO_{\lambda(r)^{-1}(U)}(D_\lambda (l')) ) $ is locally free and
$R^0\omegl_*( \calO_{\lambda(r)^{-1}(U)}(D_\lambda (l')))\otimes _{\calO_{(Q,q)}}\C(q)  \to
H^0(Z(r)_q,\calO_{Z(r)_q}(l'))$ is an isomorphism.
\end{proof}

\section{A special 1--parameter family of deformation.}\label{ss:1def}

\subsection{}  Next, we describe a special  1--parameter deformation
of a fixed resolution of a normal surface singularity $(X,o)$, what will play a crucial role in the proof of
the main Theorem \ref{th:CLB1}.

We choose any good resolution $\phi:(\tX,E)\to (X,o)$, and write
 $\cup_v E_v=E=\phi^{-1}(o)$ as above.
Since each $E_v$ is rational, a small tubular neighborhood of $E_v$ in $\tX$ can be identified with
the disc-bundle associated with the total space $T(e_v)$ of $\calO_{\bP^1}(e_v)$, where $e_v=E_v^2$.
(We will abridge $e:=e_v$.)
Recall that $T(e)$ is obtained by gluing $\C_{u_0}\times \C_{v_0}$ with $\C_{u_1}\times \C_{v_1}$
via identification  $\C_{u_0}^*\times \C_{v_0}\sim\C_{u_1}^*\times \C_{v_1}$,
$u_1=u_0^{-1}$, $v_1=v_0u_0^{-e}$, where $\C_w$ is the affine line with coordinate $w$, and $\C^*_w=\C_w\setminus \{0\}$.

Next, fix any curve $E_w$ of $\phi^{-1}(o)$  and also a {\it generic} point $P_w\in E_w$.
There exists an identification  of the tubular neighbourhood of $E_w$ via $T(e)$ such that
 $u_1=v_1=0$ is $P_w$.
By blowing up $P_w\in \tX$ we get a second resolution $\psi:\tX'\to \tX$;  the strict transforms of $\{E_v\}$'s will be denoted by $E_v'$, and the new
exceptional $(-1)$ curve by $E_{new}$. If we contract  $E_w'\cup E_{new}$ we get a cyclic quotient singularity, which is taut, hence the
 tubular neighbourhood of $E_w'\cup E_{new}$ can be identified with the tubular neighbourhood
 of the union of the zero sections in $T(e-1)\cup T(-1)$. Here we represent $T(e-1)$ as the gluing
 of $\C_{u_0'}\times \C_{v_0'}$ with $\C_{u_1'}\times \C_{v_1'}$ by
$u_1'=u_0'^{-1}$, $v_1'=v_0'u_0'^{-e+1}$. Similarly, $T(-1)$ as
$\C_{\beta}\times \C_{\alpha}$ with $\C_{\delta}\times \C_{\gamma}$ by
$\delta=\beta^{-1}$, $\gamma=\alpha\beta$. Then $T(e-1)$ and $T(-1)$ are glued
along $\C_{u_1'}\times \C_{v_1'}\sim \C_{\beta}\times \C_{\alpha}$
by
$u_1'=\alpha$, $v_1'=\beta$ providing a neighborhood of $E_w'\cup E_{new}$ in $\tX'$.
Then the neighbourhood  $\tX'$ of  $\cup_vE_v'\cup E_{new}$ will be modified by the following 1--parameter family of spaces: the neighbourhood of $\cup_vE_v' $ will stay unmodified, however
$T(-1)$, the neighbourhood of $E_{new}$ will be glued along $\C_{u_1'}\times \C_{v_1'}\sim \C_{\beta}\times \C_{\alpha}$  by
$u_1'+t=\alpha$, $v_1'=\beta$, where $t\in (\C,0)$ is a small holomorphic parameter. The smooth
complex surface obtained in this way will be denoted by $\tX'_t$, and the `moved' $(-1)$--curve
in $\tX'_t$ by  $E_{new,t}$. If we blow down $E_{new,t}$ we obtain  the surface $\tX_t$.

By construction, the  family of spaces $\{\tX'_t\}_{t\in (\C,0)}$ form a smooth 3--fold
$\widetilde{\calx}'$, together with a flat map
$\lambda':(\widetilde{\calx}',\tX') \to (\C,0)$, a $C^\infty$
trivial fibration, such that $\lambda'^{-1}(t)=\tX'_t$. Similarly,
the  family $\{\tX_t\}_{t\in (\C,0)}$ form a smooth 3--fold
$\widetilde{\calx}$, together with a flat map
$\lambda:(\widetilde{\calx},\tX) \to (\C,0)$, a $C^\infty$
trivial fibration, such that $\lambda^{-1}(t)=\tX_t$.

\begin{remark}
Such a deformation $\lambda:(\widetilde{\calx},\tX) \to (\C,0)$, reduced to some $\Gamma(\tir)$, say with
$\tir\gg 0$,
is always the pullback of a complete deformation of $\calO_{\tX}|Z(\tir)$.
Hence, if $\tX$ is generic, then the base point
$q_0$ corresponding to the fiber $\calO_{\tX}|Z(\tir)$ is in $Q\setminus \Delta(\tir)$. Since for
such $q_0$ there is a ball
$q\in U\subset Q\setminus \Delta (0,\tir)$
such that $\omegl$ simultaneously blows down to a flat family
${\mathcal X}\to U$ (cf. \ref{rem:generic}(c)),
the  deformation $\lambda:(\widetilde{\calx},\tX) \to (\C,0)$ also
 blows down to a deformation $ \calx\to (\C,0)$ of $(X,o)$.
In fact, $\lambda$  is a weak simultaneous resolution of the
(topological constant) deformation $\calx\to (\C,0)$, cf.  \cite{LaW,KSB}.
The point is that along the deformation $\lambda $ automatically
 we will have the $h^1$--stabilities
for {\it any} other finitely many restricted  natural line bundles as well, cf. Remark \ref{rem:generic}(b)
(that is, for the very same $\tX$ and its deformation $\lambda$,
the finitely many Chern classes --- whose $h^1$--stability we wish ---
can be chosen arbitrarily, depending on the geometrical situation
we treat).
\end{remark}

\section{The cohomology of restricted natural line bundles}\label{s:CNLB}

\subsection{The setup}\label{ss:setup2}
We fix a normal surface singularity $(X,o)$ and one of its good resolutions $\tX$ with
exceptional divisor $E$ and dual graph $\Gamma$.
For any integral effective cycle  $Z=Z(r)$ whose support $|Z|$ is included in
 $E$ (not necessarily the same as $E$) write $\calv(|Z|)$ for the set of vertices
 $\{v:\ E_v\subset |Z|\}$ and $\calS'(|Z|)\subset L'(|Z|) $ for the Lipman cone associated with
 the induced lattice $L(|Z|)$.
 As above, for any $l'\in L'$ we  denote the restriction of the natural
line bundle $\calO_{\tX}(l') $ to $Z$ by $\calO_Z(l')$. Denote  also by $\tilde{l}$
the cohomological restriction $ R(l')$  of $l'\in L'$ to $ L'(|Z|)$.
Recall also that  for any  $-\tilde{l}\in \calS'(|Z|)$
 one has the Abel map
$c^{\tilde{l}}:\eca ^{\tilde{l}}(Z)\to \pic^{\tilde{l}}(Z)$.

\begin{theorem}\label{th:CLB1}   Assume that  $\tX$
is generic in the sense of Definition \ref{def:GEN}. Fix also some $Z=Z(r)$ as above.
   Choose $l'=\sum_{v\in \calv}l'_vE_v \in L'$ such that
$l'_v <0$ for any $v\in\calv(|Z|)$. Then the following facts hold.

\vspace{1mm}

\noindent (I) \  Assume additionally that
$- \tilde{l}\in \calS'(|Z|)\setminus \{0\}$. Then
  the following facts are equivalent:

(a) $\calO_Z(l')\in {\rm im}(c^{\tilde{l}})$, that is,   $H^0(Z,\calO_Z(l'))_{reg}\not=\emptyset $;

(b) $c^{\tilde{l}}$ is dominant, or equivalently,
for a generic line bundle $\calL_{gen}\in \pic^{\tilde{l}}(Z)$ one has
$\calL_{gen}\in {\rm im}(c^{\tilde{l}})$ (that is,
 $H^0(Z,\calL_{gen})_{reg}\not=\emptyset $).

 (c) $\calO_Z(l')\in {\rm im}(c^{\tilde{l}})$,
 and for any $D\in (c^{\tilde{l}})^{-1}(\calO_Z(l'))$ the tangent map
 $T_Dc^{\tilde{l}}: T_D\eca^{\tilde{l}}(Z)\to T_{\calO_Z(l')}\pic^{\tilde{l}}(Z)$ is surjective.

 \vspace{1mm}

 \noindent (II)  
$h^i(Z,\calO_Z(l'))=h^i(Z,\calL_{gen})$
for a   generic line bundle  $\calL_{gen}\in \pic^{\tilde{l}}(Z)$ and $i=0,1$.
\end{theorem}
(For a remark regarding the assumptions of the theorem see \ref{rem:cohGen}(c).)
\begin{remark}\label{rem:natisgen}
The theorem  shows that if we fix $\Gamma(r)$
 then the restrictions of natural line bundles of generic singularities cohomologically
behave similarly as the generic line bundles.
This  is the main  guiding principle of the present article. This principle, in general, can be
formulated as follows. Fix some invariant associated with line bundles of resolutions with fixed graph and fixed Chern class.
Then one expects that the  invariant evaluated on the restricted  natural line bundle
 in the context of the  generic singularity  agrees with the value of the invariant evaluated
  on the   generic bundle with the same topological data
  (associated with an arbitrary fixed analytic type).

Note that by  \cite[Theorem 5.3.1]{NNI}
 the cohomology of the generic line bundles depends only on the combinatorics of $\Gamma$
 (for the formula see e.g. the introduction or (\ref{eq:cohGen})).

\end{remark}

\bekezdes {\bf Starting the proof of Theorem \ref{th:CLB1}.}\label{bek:indstep}
We use double induction over the cardinality of the subset
$\calv(|Z|)\subset \calv$  and $\sum_vr_v$.

If $|\calv(|Z|)|=1$ then $\pic^0(Z)=0$ and all line bundles with the same Chern class are isomorphic,
hence all the statements are trivially true for any $Z$ and any $l'$.
Hence let us fix some virtual support $|Z|$ and assume that
all the statements are valid for any cycle with support smaller than $|Z|$
and for any  $l'$ with the corresponding restrictions.

Next, we run induction over $\sum_{v\in \calv(|Z|)}r_v$.
Assume that $r_v\leq 1$ for all $v$. Then $\pic^{0}(Z)=0$ again and  both (I) and (II) hold.
Hence, we  assume that  (I) and (II) hold  for all cycles  with $\sum_vr_v<N$ (and any
$l'$ with the required restrictions) and we consider some
 $Z=Z(r)$ with  $\sum_vr_v=N$.

\bekezdes{\bf The first part of the proof of Theorem \ref{th:CLB1}(I).}\label{bek:kezdes}
First we verify the `easy' implications.

\vspace{1mm}

\noindent $(c)\Rightarrow(b)$ \
Since $\eca^{\tilde{l}}(Z)$ is smooth (cf. \cite[Th. 3.1.10]{NNI}), by local submersion theorem,
 if $T_Dc^{\tilde{l}}$ is surjective then the germ
$c^{\tilde{l}}:(\eca^{\tilde{l}}(Z),D)\to (\pic^{\tilde{l}}(Z),\calO_Z(l'))$ is surjective too. Since $c^{\tilde{l}}$ is
an algebraic morphism and its image contains a small analytic ball of top dimension,
$c^{\tilde{l}}$ is dominant.

\vspace{1mm}

\noindent $(b)\Rightarrow(a)$ \
Since $H^0(Z,\calL_{gen})_{reg}\not=\emptyset$,
one has $h^0(Z,\calL_{gen})\not=0$, hence by the semicontinuity of $\calL\mapsto h^0(Z,\calL)$ (cf.
\cite[Lemma 5.2.1]{NNI}) $h^0(Z,\calO_Z(l'))\not=0$ too. Next, assume that
$h^0(Z,\calO_Z(l'))_{reg}=\emptyset$, that is, there exists $v\in\calv(|Z|)$ such that
$h^0(Z,\calO_Z(l'))=h^0(Z-E_v, \calO_Z(l')(-E_v))$.
Note that $\calO_Z(l')(-E_v)|_{Z-E_v}$ is also a restricted natural line bundle, it is
$\calO_{Z-E_v}(l'-E_v)$.
Furthermore,  from  $l'_u<0$ for $u\in \calv(|Z|)$ we obtain
$(l'-E_v)_u< 0$ too.
Therefore, by the inductive step (part II)
$h^0(Z-E_v, \calO_Z(l'-E_v))=h^0(Z-E_v, \calL_{gen}(-E_v))$ and by
the assumption $h^0(Z-E_v, \calL_{gen}(-E_v))< h^0(Z,\calL_{gen})$. Thus
 $h^0(Z,\calO_Z(l'))<h^0(Z,\calL_{gen})$, a fact, which contradicts  the semicontinuity
of $\calL\mapsto h^0(Z,\calL)$.

The proof of  $(a)\Rightarrow(c)$ in (I)
 is much harder and longer, and it is the core of the present theorem.

\subsection{ The proof of $(a)\Rightarrow(c)$ in short  } \label{ss:proofAC1} The detailed proof
is presented in \ref{ss:proofAC}; in this subsection we
summarize  the main steps in order to  help the
reading of the complete  proof, though in this way inevitably some repetitions will occur.
(Since the idea of the proof -- based on the construction of the 1--parameter family -- is quite fruitful, it will be used several times in forthcoming  manuscripts as well, hence in the future work we will refer to these paragraphs as the
basic prototype.)

First we identify  $\pic^{\tilde{l}}(Z)$  with $\pic^0(Z)$
 by $\calL\mapsto \calL\otimes \calO_Z(-l')$, and  $\pic^0(Z)$  with  $H^1(Z,\calO_Z)$,
 and we replace $c^{\tilde{l}}(Z)$ with  $\widetilde{c}^{l'}(Z): \eca^{\tilde{l}}(Z)\to
H^1(\calO_Z)$.  Therefore, we wish to show that for any
 $D\in (\widetilde{c}^{l'})^{-1}(0)$ the tangent map
 $T_D\widetilde{c}^{l'}: T_D\eca^{\tilde{l}}(Z)\to T_{0}H^1(\calO_Z)$ is surjective.

 Assume that this is not happening. Then there exists a linear functional
 $\omeg\in H^1(\calO_Z)^*$, $\omeg\not=0$,
such that $\omeg|_{{\rm im}(T_D\widetilde{c}^{l'})}=0$.
This lifts to a nonzero functional $\widetilde{\omeg}$ of $H^1(\calO_{\tX})$,
which necessarily  has the form
$\widetilde{\omeg}=\langle\cdot, [\tomega]\rangle$
for some $\tomega\in H^0(\tX\setminus E,\Omega^2_{\tX})$, which
 necessarily must have a pole along some $E_w$. Using \cite{NNI} one shows that in fact
 we can choose  $E_w\subset |Z|$.
Next, we modify $\tX$ by a sequence of  blow ups.
First  we blow up $\tX$ at generic point of $E_w$ creating the new exceptional divisor
$F_1$, then we blow up a generic point of $F_1$
creating $F_2$, etc.
The sequence of $n$ such blow  ups will be denoted by $b_n:\tX_n\to \tX$, which has exceptional divisors
$\cup_{i=1}^n F_i$.
We define  $\omeg_n$ by the composition
$H^1(\calO_{b^*_n(Z)})\to H^1(\calO_Z)\stackrel{\omeg}\longrightarrow \C$
(where the first arrow is an isomorphism by Leray spectral sequence);
and similarly we set 
$\widetilde{\omeg}_n$ associated with some $\widetilde{Z}\gg0 $ (instead of $Z$).
Note  that $\tomeg_{n}\circ \,\widetilde{c}^{-F^*_n}(b_n^*(\tZ))$ corresponds to an integration
of the 2--form $b_n^*(\tomega)$
paired with divisors supported on $F_n$.
Since the pole order along $F_n$  of $b_n^*(\tomega)$
decreases by one after each blow up, after some steps $n$ it will have no pole along  $F_n$, hence
$\omeg_n\circ \widetilde{c}^{-F^*_n}(b_n^*(Z)):
\eca^{-F^*_n}(b_n^*(Z)) \to  H^1(\calO_{b^*_n(Z)}) \to \C$ is constant.
Let  $k$ be the smallest integer such that this map is constant. Then
$b_k^*(\tomega)$ has a pole of order one along $F_{k-1}$.

Next,  let $U \subset \tX_k$ be a small
 tubular neighbourhood of the exceptional curve
 $E_U:=E\cup (\cup_{i=1}^{k-1} F_i)$. Let $\Gamma_U$ be the dual graph of $E_U$.
  One  considers
the homological projection $\pi_U:L(\Gamma)\to L(\Gamma_U)$ and the
cohomological restriction $R_U:L'(\Gamma)\to L'(\Gamma_U)$
 (dual to the natural homological injection of
 cycles). Then first one identifies the germs in the corresponding spaces of effective Cartier divisors
 $(\eca^{\tilde{l}}(Z),D)\simeq
 (\eca^{b_k^*(\tilde{l})}(b_k^*(Z)),D)\simeq
 (\eca^{R_U(b_k^*(\tilde{l}))}(\pi_U(b_k^*(Z))),D)$, then one shows that
$ (\eca^{\tilde{l}}(Z),D)
 \stackrel{\widetilde{c}^{l'}}{\longrightarrow}
 H^1(\calO_Z)
 \stackrel{\omeg}{\longrightarrow} \C$ factorizes through
 $(\eca^{R_U(b_k^*(\tilde{l}))}(\pi_U(b_k^*(Z))),D)
 \stackrel{\widetilde{c}^{R_Ub^*_k(l')}}{\longrightarrow}
 H^1(\calO_{\pi_U(b^*_k(Z))})
 \stackrel{\omeg_k^U}{\longrightarrow} \C$.
 This, and the choice of $\omeg$ show that
\begin{equation*}
(\dag) \ \ \ \  \omeg_k^U\circ T_D(\widetilde{c}^{R_U(b_k^*(l'))}(\pi_U(b^*_k(Z)))=0.
\end{equation*}
Now we continue with the key construction of the proof. Using the exceptional divisors $F_{k-1}$ and $F_k$ we construct the 1--parameter family of deformation $\{\tX_{k,t}\}_t$ of $\tX_k$ (by mowing
the intersection point of $F_{k,t}$ along $F_{k-1}$), as in section \ref{ss:1def}. In this deformation one considers the universal family of natural line bundles.
Since in the central fiber $D$ is the divisor of a section of the corresponding natural line bundle,
and along the deformation the cohomology groups of the bundles are stable
(here we use the genericity),  by Lemma  \ref{lem:extension} this extends to a family of sections.
In this way we construct  a path in $\eca^{R_U(b^*_k(\tilde{l}))}(\pi_U(b^*_k(Z)))$ at $D$, $t\mapsto \gamma(t)$ (or, $\{D_t\}_t$ with $D_0=D$).
By the choice of $\omeg$ and ($\dag$) and the chain rule, 
 $\omeg\circ \widetilde{c}\circ \gamma$
must have zero derivative at $t=0$. This is valid even  for any common multiple of the divisors
$\{D_t\}_t$.
On the other hand, this derivative can be computed differently by Laufer integration.
Indeed, by taking a convenient multiple, the corresponding
powers of the members of the family of natural line bundles restricted on $U$ have the
form $\calO_{\pi_U(b^*_{k}(Z))}(\sum_v Nl'_vE_v+\ell\sum _{i=1}^{k-1} F_{i}+\ell F_{k,t})$ with $\ell\not=0$.
Here $\ell F_{k,t}\cap F_{k-1}$ is  moving divisor  along $F_{k-1}$. It paired with the
differential form of pole one by Laufer pairing has a non-trivial linear part,
cf. (\ref{eq:Tomega2}).
Hence its derivative
 at $t=0$ is nonzero, a fact which contradicts the previous statement.

\subsection{ The detailed proof of $(a)\Rightarrow(c)$  } \label{ss:proofAC}
Fix any $l^*\in L'$ and write $\overline{l}\in L'(|Z|)$ for its restriction.
Then
there is a canonical   identification of  $\pic^{\overline{l}}(Z)$  with $\pic^0(Z)$
 by $\calL\mapsto \calL\otimes \calO_Z(-l^*)$.
  Also,  $\pic^0(Z)$ identifies with  $H^1(Z,\calO_Z)$ by the inverse
of the exponential map such that $\calO_Z$ is identified with $0$.
In particular, $c^{\overline{l}}(Z): \eca^{\overline{l}}(Z)\to
\pic^{\overline{l}}(Z)$ can be identified with its composition with the above two maps,
 namely with $\widetilde{c}^{l^*}(Z): \eca^{\overline{l}}(Z)\to
H^1(\calO_Z)$.
In the sequel  $l^*$ will stay either for   $l'$ or for different cycles
of type $E_u^*$ with $E_u\in |Z|$. In this latter case,  the restriction of
$E^*_u\in L'$ is $E^*_u(|Z|)$, where this second dual is
 considered in $L'(|Z|)$. We use sometimes
 the same  notation $E^*_u$ for both of them, from the context
will be clear which one is considered.

Therefore, the wished statement $(a)\Rightarrow(c)$ transforms into
the following:
If   $D\in (\widetilde{c}^{l'})^{-1}(0)$ then the tangent map
 $T_D\widetilde{c}^{l'}: T_D\eca^{\tilde{l}}(Z)\to T_{0}H^1(\calO_Z)$ is surjective
 (under the assumptions of part (I)).

Assume that this is not the case for some $D$.
Then there exists a linear functional $\omeg\in H^1(\calO_Z)^*$, $\omeg\not=0$,
such that $\omeg|_{{\rm im}(T_D\widetilde{c}^{l'})}=0$.
During the proof we fix such a $D\in (\widetilde{c}^{l'})^{-1}(0)$
 and  $\omeg$.

First, we concentrate on $\omeg$.

\begin{lemma}\label{lem:AC1} For any $\omeg\in  H^1(\calO_Z)^*$, $\omeg\not=0$,
  there exists $E_w\subset |Z|$ such that $\,\omeg\circ\,\widetilde{c}^{-E^*_w}:
\eca^{-E^*_w}(Z)\to \C$ is not constant.
\end{lemma}
\begin{proof}
Let $\tZ=\sum_v\tir_vE_v$ be a large cycle with all $\tir_v\gg0$ $(v\in\calv$)
so that $h^1(\calO_{\tZ})=h^1(\calO_{\tX})$.
Define $\tomeg$ by the composition
$H^1(\calO_{\tZ})\stackrel{\rho}{\longrightarrow} H^1(\calO_Z)\stackrel{\omeg}{\longrightarrow} \C$.
Since $\rho$ is onto, $\tomeg\not=0$ too.
Recall that
any functional  on $H^1(\calO_{\tX})$ has the form
$\widetilde{\omeg}=\langle\cdot, [\tomega]\rangle$, cf. (\ref{eq:Stokes}),
for some $\tomega\in H^0(\tX\setminus E,\Omega^2_{\tX})$.
Since $\tomeg\not=0$ the  form necessarily must have a pole along some $E_w$.
By combination of Theorems 6.1.9(d) and 8.1.3 of \cite{NNI} we know that the kernel of $\rho$
is dual with the subspace of forms which have no pole along $|Z|$.
Therefore, $\widetilde{\omega}$  must have a pole along some
$E_w\subset |Z|$.
Since
$\eca^{-E^*_w}(Z)$ is the space of effective Cartier divisors of $\tX$ (up to the equation of $Z$),
which intersect (transversally) only $E_w$,
 again by local nature of the  integration formula,
$\tomeg\circ \widetilde{c}^{-E^*_w}(\tZ):\eca^{-E^*_w}(\tZ)\to \C$ is nonconstant, cf. (\ref{eq:Tomega2}).
But
$\omeg\circ\widetilde{c}^{-E^*_w}(Z)$ composed with
$R:\eca^{-E^*_w}(\tZ)\to \eca^{-E^*_w}(Z)$ is exactly this 
map $\tomeg\circ \widetilde{c}^{-E^*_w}(\tZ)$.
Since $R$ is surjective (cf. \cite[Theorem 3.1.10]{NNI}),
$\omeg\circ\widetilde{c}^{-E^*_w}(Z)$ is nonconstant too.
\end{proof}

\bekezdes \label{bek:szam}
Let $Z$, $\omeg$ and $E_w\subset |Z|$ be as in Lemma \ref{lem:AC1}, and $\tomega$ as in its proof,
$\widetilde{\omeg}=\langle\cdot, [\tomega]\rangle$.
We wish to modify
the resolution $\tX$ (and the space $Z$) dictated by a certain property of $\tomega$.
For this we blow up $\tX$ at generic point of $E_w$ creating the new exceptional divisor
$F_1$, then we blow up a generic point of $F_1$
creating the new exceptional divisor $F_2$, etc.
The sequence of $n$ such blow  ups will be denoted by $b_n:\tX_n\to \tX$, which has exceptional divisors
$\cup_{i=1}^n F_i$. Note also that $H^1(\calO_{b^*_n(Z)})\to H^1(\calO_Z)$ is an isomorphism (use Leray spectral sequence).
We define  $\omeg_n$ by the composition
$H^1(\calO_{b^*_n(Z)})\to H^1(\calO_Z)\stackrel{\omeg}\longrightarrow \C$.
\begin{lemma}\label{lem:AC2}
For $n$ sufficiently large the next morphism is constant:
\begin{equation}\label{eq:omega1}
\omeg_n\circ \widetilde{c}^{-F^*_n}(b_n^*(Z)):
\eca^{-F^*_n}(b_n^*(Z)) \to  H^1(\calO_{b^*_n(Z)}) \to \C.\end{equation}
 \end{lemma}
\begin{proof}
Consider $\tZ$ and the notations of the proof of Lemma \ref{lem:AC1}, and the composition
$\tomeg_n\circ \,\widetilde{c}^{-F^*_n}(b_n^*(\tZ))$, similar to (\ref{eq:omega1}),
but with $\tZ$ instead of $Z$. This for any $n$ gives the diagram

\vspace*{5mm}

\begin{equation}\label{eq:diagr2} \
\end{equation}

\vspace*{-2cm}

\begin{picture}(400,65)(-100,-10)
\put(0,40){\makebox(0,0)[l]{$
\eca^{-F^*_n}(b_n^*(\tZ)) \
 \stackrel{\widetilde{c}^{-F^*_n}}{\longrightarrow} \
 H^1(\calO_{b^*_n(\tZ)}) \
 \stackrel{\tomeg_{n}}{\longrightarrow} \ \ \C$}}

 \put(0,5){\makebox(0,0)[l]{$
\eca^{-F^*_n}(b_n^*(Z)) \
 \stackrel{\widetilde{c}^{-F^*_n}}{\longrightarrow} \
 H^1(\calO_{b^*_n(Z)}) \
 \stackrel{\omeg_{n}}{\longrightarrow} \ \ \C$}}

\put(30,22){\makebox(0,0){$\downarrow$}}\put(30,20){\makebox(0,0){$\downarrow$}}
\put(40,22){\makebox(0,0){$R_n$}}
\put(127,22){\makebox(0,0){$\downarrow$}}\put(127,20){\makebox(0,0){$\downarrow$}}
\put(185,22){\makebox(0,0){$\downarrow$}}
\put(195,22){\makebox(0,0){$\simeq$}}
\end{picture}

\noindent
Note  that $\tomeg_{n}\circ \,\widetilde{c}^{-F^*_n}(b_n^*(\tZ))$ corresponds to an integration
of the 2--form $b_n^*(\tomega)$
paired with a divisor supported on $F_n$ (cf. \ref{ss:Integ}).
Since the pole order along $F_n$  of $b_n^*(\tomega)$
decreases by one after each blow up, after some steps $n$ it will have no pole along  $F_n$, hence
$ \tomeg_{n}\circ \widetilde{c}^{-F^*_n}(b_n^*(\tZ))=
\omeg_{n}\circ \widetilde{c}^{-F^*_n}(b_n^*(Z))\circ R_n$
is constant. Since $R_n$ is surjective (see e.g. \cite[Theorem 3.1.10]{NNI}),
the statement follows.
\end{proof}
\bekezdes\label{bek:k}
In the sequel, let $k\geq 1$ be the smallest integer such that
$\omeg_k\circ \widetilde{c}^{-F^*_k}(b_k^*(Z))$ is constant.
Consider again $\tZ$  as in the proof of Lemmas
 \ref{lem:AC1} and \ref{lem:AC2}. The functionals
$\omeg_{k-1}$, $\omeg_k$,  $\tomeg_{k-1}$ and $\tomeg_{k}$  (as in \ref{bek:szam} and
 (\ref{eq:diagr2}))
 form the following commutative diagram:

\vspace*{5mm}

\begin{equation}\label{eq:diagr3} \
\end{equation}

\vspace*{-2cm}

\begin{picture}(400,100)(-100,-25)
\put(0,40){\makebox(0,0)[l]{$
 H^1(\calO_{b^*_k(\tZ)}) \
 \stackrel{\simeq}{\longrightarrow} \ \
 H^1(\calO_{b^*_{k-1}(\tZ)}) \
 \stackrel{\tomeg_{k-1}}{\longrightarrow} \ \ \ \C$}}

 \put(0,5){\makebox(0,0)[l]{$
 H^1(\calO_{b^*_k(Z)})\
 \stackrel{\simeq}{\longrightarrow} \
 H^1(\calO_{b^*_{k-1}(Z)}) \ \
 \stackrel{\omeg_{k-1}}{\longrightarrow} \ \ \ \C$}}

\put(30,22){\makebox(0,0){$\downarrow$}}\put(30,20){\makebox(0,0){$\downarrow$}}
\put(112,22){\makebox(0,0){$\downarrow$}}\put(112,20){\makebox(0,0){$\downarrow$}}
\put(182,22){\makebox(0,0){$\downarrow$}}
\put(192,22){\makebox(0,0){$\simeq$}}

\qbezier(40,50)(100,60)(165,50)
\put(165,50){\vector(4,-1){10}}
\put(100,60){\makebox(0,0){$\tomeg_{k}$}}
\qbezier(40,-8)(100,-18)(165,-8)
\put(165,-8){\vector(4,1){10}}
\put(100,-18){\makebox(0,0){$\omeg_{k}$}}
\end{picture}

\noindent
By the choice of $k$ and by the diagrams (\ref{eq:diagr2})--(\ref{eq:diagr3})
$\tomeg_{k-1}\circ \widetilde{c}^{-F^*_{k-1}}(b_k^*(\tZ))$ is nonconstant, while
$\tomeg_{k}\circ \widetilde{c}^{-F^*_k}(b_k^*(\tZ))$ is constant.
Therefore,
$b^*_{k}(\tomega)$
has a pole of order one along $F_{k-1}$. In particular, the  maps
$\eca^{-F^*_{k-1}}(b_{k}^*(V)) \to
 H^1(\calO_{b^*_{k}(V)})  \to\C$ (where $V$ is either $\tZ$ or $Z$)
depend only on the reduced structure of $ b^*_{k}(V)$ along $F_{k-1}$, and they
all can be identified with the map represented by
Laufer's  integration pairing. (For this check the integrals from \ref{ss:Integ}
for a 2--form with pole of order one.)

\bekezdes\label{bek:multstr}
In Lemma \ref{lem:AC2} and in
the discussion from \ref{bek:k} one can replace in  $\eca^{-F^*_{k-1}}$ and in
 $\eca^{-F^*_{k}}$ the cycles $F^*_{k-1}$ and $F^*_{k}$
by any multiple of them:  $N F^*_{k-1}$ and  $N F^*_{k}$
respectively, for any $N\in \Z_{>0}$.
Indeed,  the space of divisors has a natural `additive' structure, namely a dominant map
$s^{l'_1,l'_2}(V):\eca^{l_1'}(V)\times \eca^{l'_2}(V)\to \eca^{l_1'+l_2'}(V)$ which satisfies
$\widetilde{c}^{l_1'+l_2'}\circ s^{l'_1,l'_2}=\widetilde{c}^{l_1'}+
\widetilde{c}^{l_2'}$. Therefore, if for $n=k-1$ or $n=k$ the image
 ${\rm im}(\widetilde{c}^{-F^*_{n }})$ belongs to an affine subspace $A$ of
  $H^1(\calO_{b^*_n(Z)})$, then
   ${\rm im}(\widetilde{c}^{-N F^*_{n }})$ belongs to $N A:=A+\cdots +A$ too.
 In particular,
$\omeg_{k-1}\circ \widetilde{c}^{-N F^*_{k-1}}(b_k^*(Z))$ is nonconstant, while
$\omeg_{k}\circ \widetilde{c}^{-N F^*_k}(b_k^*(Z))$ is constant.
(Compare  also with the $\ell$--dependence in (\ref{eq:Tomega1}).)
   Furthermore, the discussion from \ref{bek:k} can be repeated for any $N$, the composed maps depend only on the reduced structure of $b^*_k(Z)$, hence $Z$ can be replaced by any large $\tZ$, in which case
   the composition can be computed by Laufer's  integration duality formula.

This shows that one has a factorization (where $V=\tZ$ or $Z$, and $\omeg_{V,k}=\tomeg_k$ or
 $\omeg_k$ respectively)

\vspace*{5mm}

\begin{equation}\label{eq:diagr4} \
\end{equation}

\vspace*{-1.5cm}

\begin{picture}(400,55)(-100,-10)
\put(0,40){\makebox(0,0)[l]{$
\eca^{-N F^*_{k-1}}(b_{k}^*(V)) \ \
 \stackrel{\widetilde{c}^{-N F^*_{k-1}}}{\longrightarrow} \
 H^1(\calO_{b^*_{k}(V)}) \ \
 \stackrel{\omeg_{V,k}}{\longrightarrow} \ \ \C$}}

 \put(0,5){\makebox(0,0)[l]{$
\eca^{-N F^*_{k-1}}(F_{k-1})
$}}

\put(30,22){\makebox(0,0){$\downarrow$}}\put(30,20){\makebox(0,0){$\downarrow$}}


\qbezier(85,5)(140,8)(195,25)
\put(195,25){\vector(3,1){10}}
\end{picture}

Though in (\ref{eq:diagr4})
this factorization through $\eca^{-N F^*_{k-1}}(F_{k-1})$ exists (and it is nonconstant), a factorization
through $\eca^{-N F^*_{k-1}}(F_{k-1})\to H^1(\calO_{F_{k-1}})$ definitely does not exists (because, e.g.,
$H^1(\calO_{F_{k-1}})=0$). On the other hand, a factorization through a non-trivial quotient of
$H^1(\calO_{b^*_{k}(V)})=H^1(\calO_V)$ do exists, a fact which will be crucial later.
This is what we explain next.

\bekezdes \label{bek:factorization}
 In the space of resolution $\tX_k$
 let $U \subset \tX_k$ be a small
 tubular neighbourhood of the exceptional curve
 $E_U:= E\cup (\cup_{i=1}^{k-1} F_i)$.
 Let $\Gamma_U$ be the dual graph of $E_U$.
 (Note that contracting $E_U$
 in $U$ provides a singularity with different topological type than $\Gamma$, one of its
 dual graphs is $\Gamma_U$.)
 One can restrict sheaves/bundles from $\tX_k$ to $U$. At  cycle level one has
 the homological projection
 $\pi_U(\sum_vn_vE_v+\sum_{i=1}^km_iF_i):=\sum_vn_vE_v+\sum_{i=1}^{k-1}m_iF_i$.
 One also has the cohomological restriction $R_U:L'(\Gamma)\to L'(\Gamma_U)$
 (dual to the natural homological injection of
 cycles);   e.g.
 the restriction $R_U(F^*_{k-1})$ of $F^*_{k-1}$  is the antidual rational cycle
 $F^*_{k-1}(\Gamma_U)$
 associated with $F_{k-1}$
 in the lattice of $\Gamma_U$.
 Then, for both $V=\tZ$ or $Z$, one has the natural injection
 (which, for $V=\tZ$ and $Z$ fit in a commutative diagram):
 $\eca^{-N F^*_{k-1}}(b^*_k(V))$ is a Zariski open set in
  $\eca^{-N R_U(F^*_{k-1})}(\pi_U(b^*_k(V)))$.
Indeed, both of them depend only on the multiplicity $m_{k-1}$ of $F_{k-1}$ in
$b^*_k(V)$ and $\pi_U(b^*_k(V))$ (which are equal),
the second  set contains divisors up to the equation of $m_{k-1}F_{k-1}$
supported on $F_{k-1}\setminus F_{k-2}$ with total multiplicity
$N$, while in the first set consists of those divisors of the second set
whose support does not contain   $F_{k-1}\cap F_k$.

 On the other hand, the natural epimorphism
  $\rho_V: H^1(\calO_{b^*_k(V)})\to H^1(\calO_{\pi_U(b^*_k(V))})$ usually is not a monomorphism. However, one has the following fact.

\begin{lemma}\label{lem:restrict} 
$\omeg_{V,k}:H^1(\calO_{b^*_k(V)})\to \C$ factors through $\rho_V
 :H^1(\calO_{b^*_k(V)})\to H^1(\calO_{\pi_U(b^*_k(V))})$.
\end{lemma}

\begin{proof}
First, we concentrate on the map $\widetilde{c}^{-F^*_k}:\eca^{-F^*_k}(b^*_k(V))\to H^1(\calO_{b^*_k(V)})$.
Let $A$ be the smallest affine subspace
of $H^1(\calO_{b^*_k(V)})$  which contains ${\rm im}(\widetilde{c}^{-F^*_k})$, and let $A_0$
be the parallel linear subspace of the same dimension.  As above, we denote the sum
$A+\cdots +A $ ($m$ times) by $mA$, clearly all of these affine subspaces
have the same dimension, and are parallel
to each other. Next, consider also the
`multiples'  $\widetilde{c}^{-mF^*_k}:\eca^{-mF^*_k}(b^*_k(V))\to H^1(\calO_{b^*_k(V)})$ (cf.
\cite[\S 6]{NNI}, or see \ref{bek:multstr}).
Therefore,
${\rm im}(\widetilde{c}^{-mF^*_k})\subset mA$, and in fact, by \cite[Theorem 6.1.9]{NNI}, for $m\gg 0$, they agree. Furthermore, by the same theorem, $A_0=\ker( \rho_V)$.

By the choice of $k$, $\omeg_{V,k}$ restricted on the image of $\widetilde{c}^{-F^*_k}$ is constant,
which means that $\omeg_{V,k}|A$ is constant, or
$A_0\subset \ker(\omeg_{V,k})$.  Hence $\ker(\rho_V)\subset \ker(\omeg_{V,k})$, and $\omeg^U_{V,k}$
with $\omeg^U_{V,k}\circ \rho_V=\omeg_{V,k}$   exists.
\end{proof}

This lemma has the following geometric interpretation.
If $\omeg_{V,k}=\langle\cdot,[b^*_k\tomega]\rangle$ (at the level of $V$ or
$\tX_k$),
then $\omeg_{V,k}^U=\langle\cdot,[b^*_k\tomega|_U]\rangle$  at the level of $U$.
The form $b^*_k\tomega|_U$ again has order one along $F_{k-1}$ and
 all the local integration formulas along $E_U$ are the same.

\bekezdes\label{bek:D} Next, we concentrate on the divisor $D\in \eca^{\tilde{l}}(Z)$ and on the line bundle $\calO_Z(l')=\calO_Z(D)$.
As the center of blow up of $b_1$ is generic on $E_w$,
we can assume that it is not in the support of $D$. This guarantees that the
divisor $D$ lifts canonically into any of the spaces $\eca^{b_k^*(\tilde{l})}(b_k^*(Z))$
(still denoted by $D$),
and the germs
$(\eca^{\tilde{l}}(Z) ,D)$ and $(\eca^{b_k^*(\tilde{l})}(b_k^*(Z)),D)$ are canonically isomorphic.

Furthermore, this germ is preserved under the restriction to $U$ (see also the argument
from \ref{bek:factorization}),
hence all these facts together with the existence of factorization from Lemma \ref{lem:restrict} can be inserted in the following commutative diagram:

\vspace*{5mm}

\begin{equation}\label{eq:diagr1} \
\end{equation}

\vspace*{-2cm}

\begin{picture}(400,100)(-100,-10)
\put(0,75){\makebox(0,0)[l]{$
\ \ \ (\eca^{\tilde{l}}(Z),D) \ \ \ \ \ \ \ \
 \stackrel{\widetilde{c}^{l'}}{\longrightarrow} \ \ \ \
 H^1(\calO_Z)\ \  \ \ \ \
 \stackrel{\omeg}{\longrightarrow} \ \ \ \ \ \, \C$}}

 \put(-5,40){\makebox(0,0)[l]{$
(\eca^{b_k^*(\tilde{l})}(b_k^*(Z)),D) \ \
 \stackrel{\widetilde{c}^{b^*_k(l')}}{\longrightarrow} \ \
 H^1(\calO_{b^*_k(Z)}) \ \ \
 \stackrel{\omeg_k}{\longrightarrow} \ \ \ \ \ \ \C$}}

\put(50,55){\makebox(0,0){$\uparrow$}}
\put(140,57){\makebox(0,0){$b'_n$}}
\put(158,55){\makebox(0,0){$\simeq$}}\put(58,55){\makebox(0,0){$\simeq$}}
\put(150,55){\makebox(0,0){$\uparrow$}}
\put(230,55){\makebox(0,0){$\uparrow$}}
\put(237,55){\makebox(0,0){$\simeq$}}

\put(50,20){\makebox(0,0){$\downarrow$}}
\put(140,22){\makebox(0,0){$\rho_Z$}}
\put(58,20){\makebox(0,0){$\simeq$}}
\put(150,20){\makebox(0,0){$\downarrow$}}
\put(230,20){\makebox(0,0){$\downarrow$}}
\put(237,20){\makebox(0,0){$\simeq$}}

 \put(-43,5){\makebox(0,0)[l]{$
(\eca^{R_U(b_k^*(\tilde{l}))}(\pi_U(b_k^*(Z))),D) \
 \stackrel{\widetilde{c}^{R_Ub^*_k(l')}}{\longrightarrow} \
 H^1(\calO_{\pi_U(b^*_k(Z))}) \
 \stackrel{\omeg_k^U}{\longrightarrow} \ \, \C$}}

\end{picture}

This diagram shows that $\omeg_k\circ T_D( \widetilde{c}^{b_k^*(l')}(b^*_k(Z)))=0$ and also
\begin{equation}\label{eq:deriv}
\omeg_k^U\circ T_D(\widetilde{c}^{R_U(b_k^*(l'))}(\pi_U(b^*_k(Z)))=0.
\end{equation}
\bekezdes\label{bek:wk}
On $b^*_k(Z)$ now  we have the pullback line bundle
$b^*_k(\calO_Z(l'))=b^*_k(\calO_Z(D))=\calO_{b^*_k(Z)}(D)$.
 \begin{lemma}\label{lem:pullbacknat}
$b^*_k(\calO_{\tX}(l'))=\calO_{\tX_k}(b^*_k(l'))$, that is, the pullback of the natural line bundle
$\calO_{\tX}(l')$
 is the natural line bundle associated with the Chern class $b^*_k(l')$.
 Therefore,
 $b^*_k(\calO_{Z}(l'))=\calO_{\tX_k}(b^*_k(l')|_{b^*_k(Z)})$
  (which will be denoted by $\calO_{b^*_k(Z)}(b^*_k(l'))$).
\end{lemma}
\begin{proof}
A bundle is natural if one of its power has the form $\calO(l)$ for some integral cycle $l$.
In this case the Chern classes of the two bundles agree. Furthermore,
if $nl'$ is integral for certain $n\in \Z_{>0}$, then
$b^*_k(\calO_{\tX}(l')^{\otimes n})=\calO_{\tX_k}(b_k^*(nl'))$, hence
$b^*_k(\calO_{\tX}(l'))$ is natural with Chern class $b_k^*(l')$.
\end{proof}

After all these preparations, we start with the key construction of the proof. We will construct a path in $\eca^{R_U(b^*_k(\tilde{l}))}(\pi_U(b^*_k(Z)))$ at $D$, $t\mapsto \gamma(t)$ (or, $\{D_t\}_t$ with $D_0=D$)
 with the following properties.
Firstly, by the choice of $\omeg$ and (\ref{eq:deriv})
 $\omeg\circ \widetilde{c}\circ \gamma$
must have zero derivative at $t=0$.  On the other hand, we will compute by integration explicitly
$\omeg\circ \widetilde{c}\circ \gamma$
and we will show that its  linear part is nontrivial, hence its derivative
 at $t=0$ is nonzero, a fact which leads to a contradiction.

 The local  path of divisors will be constructed via a deformation, based on section \ref{ss:1def}.

 \bekezdes\label{bek:specialdef}
 {\bf A special deformation of the analytic structure of $\calO_{\tX_k}$.} \

Let $(\tX_k, E\cup\cup_{i=1}^{k}F_i)$ be the resolution as in \ref{bek:szam},
with the choice  of $k$ as in \ref{bek:k}. Here we concentrate on the exceptional components
$F_{k-1}$ and $F_k$, where $F_k$ is obtained by blowing up a generic point $P$. (If $k=1$ then $F_{k-1}=E_w$.)
Then for the pair $(F_{k-1},F_k)$ we apply the construction of section \ref{ss:1def}, that is,  we move
$F_k$ and its intersection point with $F_{k-1}$ locally along $F_{k-1}$. In this way we obtain
a 1--parameter family of deformations of the resolution $\tX_k$, denoted by
$\omegl_k:(\widetilde{\calx}_k,\tX_k)\to (\C,0)$, with fibers $\tX_{k,t}$.
In $\tX_{k,t}$ the exceptional curve has components
$E\cup\cup_{i=1}^{k-1}F_i\cup F_{k,t}$. If we blow down the $F$--type curves in $\tX_{k,t}$ we get
a resolution $\tX_t$, they form a family $(\widetilde{\calx},\tX)$.
If we contract all the exceptional curves we get a family of   singularities
$\{(X_t,o)\}_t$. Since the analytic structure we started with is generic,
the  geometric genus $h^1(\calO_{\tX_{k,t}})$ stays constant
and the deformation blows down to  a deformation
$(\calx,X)\to (\C,0)$ with fibers $X_t$ (cf. \ref{ss:1def}).
We denote the contraction $\widetilde{\calx}_k\to \widetilde{\calx}$ by the same symbol $b_k$.

We assume that the base space of $\omegl$ is so small that the universal map $(\C,0)
\to Q$ to the base space of  a complete deformation omits the discriminant $\Delta(\tir)$; this fact is guaranteed by the choice of the generic structure of the singularity.

Therefore, for the very same $l'\in L'$ (which provides
 the bundle $\calO_Z(l')$)
we can consider the universal line bundles constructed in Lemma \ref{lem:natline}, namely
$\calO_{\widetilde{\calx}_k}(b^*_k(l'))\in \pic(  \widetilde{\calx}_k)$ and
$\calO_{\widetilde{\calx}}(l')\in \pic(  \widetilde{\calx})$. By similar argument as in Lemma
\ref{lem:pullbacknat} we have $b_k^*(\calO_{\widetilde{\calx}}(l'))=
\calO_{\widetilde{\calx}_k}(b^*_k(l'))$.
The  restriction to the fibers of the deformations
are the natural line bundles of the fibers.


Corresponding to the irreducible exceptional curves
 $\{E_v\}_v$ and $\{F_i\}_{i=1}^k$ in $\tX_{k}$
we have the irreducible exceptional surfaces
$\{\cale_v\}_{v}$ and $\{\calf_i\}_{i=1}^k$ in $\widetilde{\calx}_{k}$.
(Here $(\calf_n)_t=F_n$ for $n<k$ but $(\calf_k)_t=F_{k,t}$.)
If $Z=\sum_{v}r_vE_v$ then
$b^*_k(Z)=\sum_{v}r_vE_v+r_w\sum_{i=1}^kF_i$.
Let  we set
$b^*_k(\calz)=\sum_{v\in \calv}r_v\cale_v+r_w\sum_{i=1}^k\calf_i$. Then we restrict
$\calO_{\widetilde{\calx}_{k}}(b^*_k(l'))$ to $b^*_k(\calz)$ and we get
$\calO_{b^*_k(\calz)}(b^*_k(l'))\in \pic(b^*_k(\calz))$.

 Let  $\omegl: b^*_k(\calz)\to (\C,0)$ be the projection of the deformation.
The central fiber is  $\calO_{b^*_k(Z)}(b_k^*(l'))$.
In particular, over
 $t=0$ the bundle $\calO_{b_k^*(Z)}(b_k^*(l'))$ has a global section $s$ whose divisor is  $D$
 (by the definition of $D$  from \ref{ss:proofAC} and identification (\ref{eq:diagr1})).
Then Lemma \ref{lem:extension} implies  the following fact.

\begin{lemma}\label{lem:extension2} There exists an extension
${\mathfrak s}\in H^0(b_k^*(\calz), \calO_{b^*_k(\calz)}(b^*_k(l')))$ of
 $s\in H^0(b_k^*(Z), \calO_{b^*_k(Z)}(b^*_k(l')))$ such that  ${\mathfrak s}_0=s$.
 Furthermore,  ${\mathfrak s}_t$ has no fixed component either.
\end{lemma}

Let $D_t$ be the restriction of the divisor of ${\mathfrak s}$ to the fiber over $t$.

Since the support of $D=D_0$  is disjoint with the center of
$b_1$, the same is true for each $D_t$ (for $|t|\ll 1$).
Hence, in this way we get a  path  germ
$\gamma$  with $\gamma(t)\in \calO_{b^*_k(\calz)_t}(D_t)=\calO_{b^*_{k,t}(Z)}(D_t)=\calO_{b^*_{k,t}(Z)}(b^*_{k,t}(l'))$, where $b_{k,t}$ is the
contraction/blow up $\tX_{k,t}\to \tX_{t}$.

Note also that in the cycles $b_{k,t}^*(Z)$ the curve $F_{k,t}$ (with its stable multiplicity) is `moving' along the deformation,
the other components with their multiplicities  are stable, and the divisors $D_t$ are supported
by this stable part (but they might  move).  More precisely,
 by the construction from \ref{bek:specialdef} we obtain that
$\pi_U(b^*_k(\calz)_t)$ is $t$--independent, and it equals $\pi_U(b^*_k(Z))$.
(It is worth to mention that $\pi_U(b^*_k(Z))$ is not the same as $b^*_{k-1}(Z)$, they differ even topologically at Euler number level.)

Then, by the choice of $\omeg$  and  $D$ and the chain rule
(compare also with (\ref{eq:diagr1}) and (\ref{eq:deriv}):
\begin{equation}\label{eq:dergamma}\begin{split}
&\frac{d}{dt}\Big|_{t=0}(\omeg_k\circ \widetilde{c}^{b^*_{k,t}(l')}(b^*_{k,t}(Z))
( \gamma(t))=
\frac{d}{dt}\Big|_{t=0}(\omeg_k^U\circ \widetilde{c}^{R_U(b^*_{k,t}(l'))}
(\pi_U(b^*_{k,t}(Z)))( \gamma(t))\\ =&
T_D(\omeg_k^U\circ \widetilde{c}^{R_U(b^*_{k}(l'))}(\pi_U(b^*_{k}(Z))) (\frac{d\gamma}{dt}\Big|_{t=0})
=\omeg_k^U\circ T_D(\widetilde{c}^{R_U(b^*_{k}(l'))}(\pi_U(b^*_{k}(Z))) (\frac{d\gamma}{dt}\Big|_{t=0})
=0.\end{split}
\end{equation}
The same is valid if we replace the family $D_t$ by any of its  multiple $N\cdot D_t$.

\bekezdes \label{bek:sum} Let us summarize what we have. On each $b^*_{k,t}(Z)$
we can consider the restricted natural line bundle
$\calO_{b^*_{k,t}(Z)}(b^*_{k,t}(l'))$. Then, if we take its restriction to $U$,
namely $\calO_{b^*_{k,t}(Z)}(b^*_{k,t}(l'))|_U\in \pic(\pi_U(b^*_k(Z)))$
and we shift it back with the natural line bundle
$\calO_{\pi_U(b^*_k(Z))}(R_U(b^*_k(l')))^{-1}$ we get a path in
$\pic^0(\pi_U(b^*_k(Z)))=H^1(\calO_{\pi_U(b^*_k(Z))})$, whose differential at $t=0$
is in the kernel of $\omeg^U_k$.

Now, let us compute these objects directly, in fact, for a certain  $N$--multiple of
the corresponding bundles. Let $N$ be an integer so that
$Nl'=\sum_v Nl'_vE_v$ is an integral cycle and write $\ell:=Nl'_w$. Then,
$Nb^*_k(l')=\sum_v Nl_v'E_v+\ell \sum _{i=1}^k F_i$. Furthermore,
$(\calO_{b^*_{k,t}(Z)}(b^*_{k,t}(l')))^N$, being natural with integral Chern class,
should equal
$\calO_{b^*_{k,t}(Z)}(\sum_v Nl'_vE_v+\ell\sum _{i=1}^k F_{i,t})$ and
its restriction to $U$ is
$\calO_{\pi_U(b^*_{k}(Z))}(\sum_v Nl'_vE_v+\ell\sum _{i=1}^{k-1} F_{i}+\ell F_{k,t})$.
By the same reason, $\calO_{\pi_U(b^*_k(Z))}(R_U(b^*_k(l')))^{-N}$
is $\calO_{\pi_U(b^*_{k}(Z))}(\sum_v Nl'_vE_v+\ell\sum _{i=1}^{k} F_{i})$.
Hence, the $N$--multiple of the path is
$\calO_{\pi_U(b^*_{k}(Z))}(\ell(P_t-P))$, where  $P_t=F_{k,t}\cap F_{k-1}$,
$P=F_k\cap F_{k-1}$ as above. By assumption on $l'_w$ we have $\ell\not=0$.

That is,  $\calO_{\pi_U(b^*_{k}(Z))}(\ell P_t-\ell P)$ is a path in
$ H^1(\calO_{\pi_U(b^*_k(Z))})$ and (\ref{eq:dergamma}) reads as
\begin{equation}\label{eq:dergamma3}
\frac{d}{dt}\Big|_{t=0}(\omeg^U_{k} (\calO_{\pi_U(b^*_{k}(Z))}(\ell P_t-\ell P)) =0.
\end{equation}
Next we compute the left hand side of (\ref{eq:dergamma3}) in a  different way.

By Lemma \ref{lem:restrict} (and comment after it)
$\omeg^U_k=\langle \cdot, [b^*_k\tomega |_U\rangle$, and the form $b^*_k\tomega |_U$
has a pole of order one along $F_{k-1}$. Moreover, $P$ is a generic point of
$F_{k-1}$ and in a local neighborhood $B$ of $P$ in local coordinates $(u,v)$ one has
$F_{k-1}\cap B=\{u=0\}$, $P_t=\{v+t=0\}$. Hence (\ref{eq:Tomega2}) with $o=1$ reads as
\begin{equation}\label{eq:dergamma4}
\omeg^U_{k} (\calO_{\pi_U(b^*_{k}(Z))}(\ell P_t-\ell P)) = t\ell c +\{\mbox{higher order terms}\}\ \ \  (c\in \C^*),
\end{equation}
whose derivative at $t=0$  is non--zero. This contradicts (\ref{eq:dergamma3}).

\subsection{The proof of part (II)}\label{ss:PartII}
 Note that the equalities for $i=0$ and $i=1$ are equivalent by Riemann--Roch.
 We will prove (II) in three steps.

 \bekezdes\label{ss:PartII1} {\bf The proof of part (II), case 1.}\
Assume that $l'_v<0$ for any $v\in \calv(|Z|)$ and   $-\tilde{l}\in\calS'(|Z|)\setminus \{0\}$.

Then part (I) --- already proved --- can be applied.

First assume that the equivalent assumptions {\it (a)-(b)-(c)} of (I) are satisfied.
Then by \cite[Th. 4.1.1]{NNI} $h^1(Z,\calL_{gen})=0$. Hence we have to show that
 $h^1(Z,\calO_Z(l'))=0$ too. Choose an element $s\in H^0(Z,\calO_Z(l'))_{reg}$ with divisor
 $D$  and consider the  exact sequence of sheaves
 $0\to \calO_Z\stackrel{\times s}{\longrightarrow} \calO_Z(l') \to \calO_D(D)\to 0$ (where the second morphism is  multiplication by $s$).

 Then one has the cohomology exact sequence
 $$H^0(Z,\calO_Z(l'))\to \calO_D(D)\stackrel{\delta}{\longrightarrow} H^1(\calO_Z)\to
 H^1(Z,\calO_Z(l'))\to 0.$$
 Then $\delta$ can be identified with $T_D(c^{\tilde{l}})$ (see \cite[Prop. 3.2.2]{NNI}, or
  \cite[p. 164]{MumfordCurves}, \cite[Remark 5.18]{Kl}, \cite[\S 5]{Kleiman2013}).
 Since $T_D(c^{\tilde{l}})$ is onto by (I)(c),  $h^1(Z,\calO_Z(l'))=0$ follows.

Next, assume that 
the equivalent assumptions of (I) are not satisfied.
That is, $H^0(Z,\calO_Z(l'))_{reg}=H^0(Z,\calL_{gen})_{reg}=\emptyset$. These facts read as
$h^0(Z,\calO_Z(l'))=\max_v\{h^0(Z-E_v,\calO_Z(l'-E_v))\}$ and
$h^0(Z,\calL_{gen})=\max_v\{h^0(Z-E_v,\calL_{gen}(-E_v))\}$.
But, by induction (applied for part (II) similarly as in the proof of case $(b)\Rightarrow(c)$ in
\ref{bek:kezdes}, see also  \ref{bek:indstep})
$\max_v\{h^0(Z-E_v,\calO_Z(l'-E_v))\}=
\max_v\{h^0(Z-E_v,\calL_{gen}(-E_v))\}$, hence
$h^0(Z,\calO_Z(l') )=h^0(Z,\calL_{gen})$ follows too.

 \bekezdes \label{ss:PartII2} {\bf The proof of part (II), case 2.}\
Assume that $l'_v<0$ for any $v\in \calv(|Z|)$ and  $\tilde{l}=0$.
(If this happens then necessarily $|Z|<E$. Recall also that $\calO_Z(l')$ is the restriction
of the natural line bundle $\calO_{\tX}(l')$ to $Z$.)

If $h^1(\calO_Z)=0$ then $\calL_{gen}=\calO_Z(l')$, hence the statement follows.
If $h^0(\calO_Z(l'))=0$ then by the semicontinuity of $\calL\mapsto h^0(Z,\calL)$ (cf.
\cite[Lemma 5.2.1]{NNI}) $h^0(\calL_{gen})=0$ too.

In the sequel we assume that $h^1(\calO_Z)\not=0$ and $h^0(\calO_Z(l'))\not=0$.

Assume that $H^0(Z,\calO_Z(l'))_{reg}\not=\emptyset$, that is, $\calO_Z(l')$ has a section without
fixed components. But, then by Chern class computation, this section  has no zeros,
hence
$\calO_Z(l')=\calO_Z$, see also (\ref{eq:Chernzero}).

{\it We claim that this identity $\calO_Z(l')=\calO_Z$ cannot happen for generic $(X,o)$. }

The argument runs similarly as the proof of $(a)\Rightarrow(c)$ in (I).

Since $h^1(\calO_Z)\not=0$ we can choose a nonzero functional $\omega\in H^1(\calO)^*$ for which
we can repeat the arguments from \ref{ss:proofAC}. In particular,
there exists $E_w\subset |Z|$ which satisfies Lemma \ref{lem:AC1}, we can consider the sequence of
blow ups as in \ref{bek:szam}, and we can choose $k$ as in \ref{bek:k}. Finally we consider the
deformation of singularities as in \ref{bek:specialdef}.
In this way we get a family of restricted line bundles
$\calO_{b^*_{k,t}(Z)}(b^*_{k,t}(l'))$, so that for $t=0$ the corresponding bundle is the
trivial one. We wish to show that for generic $t$ the corresponding term cannot be the trivial
bundle. Indeed, as in (\ref{eq:dergamma4})
we get that  $t\mapsto
\calO_{b^*_{k,t}(Z)}(b^*_{k,t}(l'))|_U\in \pic(\pi_U(b^*_k(Z)))$ is not constant.
This implies that
the path $t\mapsto b^*_{k,t}(\calO_Z(l'))=\calO_{b^*_{k,t}(Z)}(b^*_{k,t}(l'))$ cannot
give for all $t$ the trivial bundle  either
since otherwise its  restriction to $\pi_U(b^*_k(Z))$ would be  constant (since the restriction of the structures sheaf is the $t$--independent constant structure sheave).
In particular, for generic $t$ we have $\calO_{Z_t}(l')\not= \calO_{Z_t}$.


However, we can prove that in this situation necessarily
$h^1(\calO_{Z_t}(l'))< h^1(\calO_{Z_t})$ for generic $t$ (though the Chern classes agree),
hence $t=0$ is a jumping discriminant point of $l'\mapsto h^1(\calO_{Z_t}(l'))$, a fact which contradict the genericity.

Indeed, since $\calO_{Z_t}(l')\not=\calO_{Z_t}$
for generic $t$ (and $H^1(\calO_{Z_t})$ is constant nonzero),
 $\calO_{Z_t}(l')$ must have fix components
 (use $c_1(\calO_{Z_t}(l'))=0$ and  (\ref{eq:Chernzero})).
Let $E_u\in |Z|$ be a fix  component. Then $H^0(Z_t,\calO_{Z_t})\to H^0(E_u,\calO_{Z_t})=\bC$ is
surjective, while $H^0(Z_t,\calO_{Z_t}(l') )\to H^0(E_u,\calO_{Z_t}(l'))=\bC$ is zero.
Since their kernels have the same $h^0$ by the inductive step,
$h^0(\calO_{Z_t}(l'))< h^0(\calO_{Z_t})$, hence
the inequality follows by Riemann--Roch.
This proves the claim.

\vspace{2mm}

After  this discussion we can assume that
 $h^1(\calO_Z)\not=0$, $h^0(\calO_Z(l'))\not=0$, but $H^0(Z,\calO_Z(l'))_{reg}=\emptyset$.
 By (\ref{eq:Chernzero})  $\calL_{gen}\not=\calO_Z$ (since $\pic^0(\calO_Z)\not=0$), hence
 $H^0(Z,\calL_{gen})_{reg}=\emptyset$ too. Then we proceed as in the last paragraph  of
 \ref{ss:PartII1}, induction shows that
 $h^0(Z,\calO_Z(l') )=h^0(Z,\calL_{gen})$.

 \bekezdes\label{ss:PartII3} {\bf The proof of part (II), case 3.} \
Finally,  assume that $l'_v<0$ for all $v\in \calv(|Z|)$, and
$-\tilde{l}\not \in\calS'(|Z|)$.
Then there  exists $E_v$ in the support of $Z$
such that $(l',E_v)=(\tilde{l},E_v)<0$. Hence for any $\calL\in \pic^{\tilde{l}}(Z)$ the exact sequence
$0\to \calL(-E_v)|_{Z-E_v}\to \calL\to \calL|_{E_v}\to 0$  and vanishing
$H^0(\calL|_{E_v})=0$ give
$h^0(Z-E_v,\calL(-E_v))=h^0(Z,\calL)$. By this step we replaced the Chern class $\tilde{l}$ by
$\tilde{l}-E_v$.
After finitely many such steps we necessarily get a  new Chern class
in the corresponding Lipman cone (see e.g. \cite[Prop. 4.3.3]{trieste}).
Hence, in this way we reduced this third case to the first two cases.

\section{Applications. Analytic invariants}\label{s:appli}

\subsection{} In this section we will fix a resolution graph $\Gamma$
(hence, the lattice $L$ associated  with it as well),
and we treat singularities $(X,o)$, together with their resolution $\tX$ whose dual graph is $\Gamma$.
The goal is to list  some consequences of Theorem \ref{th:CLB1}: hence we will assume that $\tX$ is generic,
and we will provide combinatorial expressions for several analytic invariants in terms of $L$.
We will use the notations from the setup of \ref{ss:setup2}.

The first group of results provides topological formulae for the
{\bf cohomology of}  certain {\bf natural line bundles} over an arbitrary $Z>0$.

\begin{remark}\label{rem:cohGen} \ (a)
By \cite[Theorem 5.3.1]{NNI} for any $l'\in L'$ and $\calL_{gen}$ generic in $\pic^{R(l')}(Z)$
\begin{equation}\label{eq:cohGen}
h^1(Z,\calL_{gen})=\chi(-l')-\min _{0\leq l\leq Z, l\in L}\, \{\chi(-l'+l)\}.
\end{equation}
In particular, if $l'=\sum_{v\in \calv}l'_vE_v \in L'$ satisfies
$l'_v <0$ for any $v\in\calv(|Z|)$ and $\tX$ is generic then Theorem \ref{th:CLB1} gives
the following topological characterization for the cohomology of $\calO_Z(l')$
\begin{equation}\label{eq:cohNat}
h^1(Z,\calO_Z(l'))=\chi(-l')-\min _{0\leq l\leq Z, l\in L}\, \{\chi(-l'+l)\}.
\end{equation}
This will be extended in Theorem \ref{th:CohNat} for a larger family of  $l'$--values.

(b) Note that the identity $h^1(Z,\calO_Z(l'))=h^1(Z,\calL_{gen})$ (hence (\ref{eq:cohNat}) too)
is not valid for any $l'$ (that is, without some negativity condition regarding the coefficients of
 $l'$).
Indeed, assume e.g. that $|Z|=E$ and all the coefficients of $Z$ are very large,  and $l'=0$.
Then using the quadratic form of $\chi$ one has
$\min _{0\leq l\leq Z, l\in L}\, \{\chi(l)\}=\min _{ l\in L_{\geq 0}}\, \{\chi(l)\}$, hence
 $h^1(Z, \calL_{gen})=-\min _{l\in L_{\geq 0}}\, \{\chi(l)\}$ by (\ref{eq:cohGen}).
 But $h^1(Z,\calO_Z)=1-\min _{l\in L_{\geq 0}}\, \{\chi(l)\}$ whenever $(X,o)$ is not rational,  see Corollary \ref{cor:CohNat}.

 (c) Recall that if $-l'\in \calS'\setminus \{0\}$ then all the coefficients $l'_v$ of $l'$ are
 strict negative. However, if the support of $|Z|$ is strict smaller than $E$, then
  $-R(l')\in \calS'(|Z|)\setminus \{0\}$ does not necessarily imply that $l'_v <0$ for $v\in \calv(|Z|)$.
  (Take e.g. $Z=E_v$ a $(-2)$--curve, choose $E_u$ an adjacent vertex with it and set
  $l'=E_v+3E_u$. Then $-R(l')\in \calS'(E_v)\setminus \{0\}$ however $l'_v=1$.)
 \end{remark}
\bekezdes\label{bek:Lauferseq} {\bf The setup for generalization.}
 We construct the following `Laufer type computation sequence'  (see e.g. \cite{Laufer72} or
 \cite[Prop. 4.3.3]{trieste}).
 We start with a class $l'\in L'$ and an effective  cycle $Z$ with $|Z|\subset E$.
 Let $\tilde{l}\in L'(|Z|)$ be the restriction of $l'$ as in Theorem \ref{th:CLB1}.

 Assume that  $-\tilde{l}\not\in \calS'(|Z|)$. Then there exists
$E_w\subset |Z|$ so that $(l',E_w)<0$.
Then, for both line bundles $\calL=\calL_{gen}$ and $\calL=\calO_Z(l')$ of $\pic^{\tilde{l}}(Z)$
one can consider the exact sequence $0\to \calL(-E_w)|_{Z-E_w}\to \calL\to \calL|_{E_w}\to 0$,
hence $h^0(\calL(-E_w)|_{Z-E_w})=h^0(\calL)$. Hence whenever
$h^0(\calO_Z(l'-E_w)|_{Z-E_w})=h^0(\calL_{gen}(-E_w)|_{Z-E_w})$ one also has
$h^0(\calO_Z(l'))=h^0(\calL_{gen})$.

Let us construct the following sequence of pairs
$(l'_k,Z_k)_{k=0}^t$. By definition, $(l'_0,Z_0)=(l',Z)$
the objects we started with. If   $-\tilde{l}=-R(l')\not\in \calS'(|Z|)$, then define
$(l'_1,Z_1):=(l'-E_w,Z-E_w)$ for some $E_w\subset |Z|$ with $(E_w,l')<0$. If  $-\tilde{l}_1:=-R(l'_1)\not\in \calS'(|Z_1|)$
we repeat the procedure, otherwise we stop. After finitely many steps necessarily
 $-\tilde{l}_t:=-R(l'_t)\in \calS'(|Z_t|)$ (here $Z_t=0$ is also possible).
(The choice of the sequence is not unique, however
by similar argument as in  \cite{Laufer72} or
 \cite[Prop. 4.3.3]{trieste}) one can show that the last term $(l'_t,Z_t)$ of the sequence is independent of all the choices: it is the unique $(l'-D,Z-D)$ with $D$ minimal such that
 $Z\geq D\geq 0$, $D\in L$, and $-(l'-D)\in \calS'(|Z-D|)$.)

\begin{theorem}\label{th:CohNat} Assume that $\tX$ is generic with fixed dual graph $\Gamma$, and we
choose  an effective   cycle  $Z$  and  $l'\in L'$.
Assume that the last term $(l'_t,Z_t)$ of the Laufer type computation sequence $\{(l'_k,Z_k)\}_{k=0}^t$
has the following property: if $l'_t=\sum_v l'_{t,v}E_v$, then $l'_{t,v}<0 $ for  any
$v\in \calv(|Z_t|)$. Then $h^i(Z,\calO_Z(l'))=h^i(Z,\calL_{gen})$
for a generic line bundle $\calL_{gen}\in\pic^{\tilde{l}}(Z)$ ($i=0,1$), i.e.
(\ref{eq:cohNat}) holds.
\end{theorem}

\begin{proof}
Use Theorem \ref{th:CLB1}(II) and the discussion from \ref{bek:Lauferseq}.
\end{proof}

\begin{example}\label{ex:cohNat}
Let $\tX$ be generic, $Z$ an effective cycle and $l'\in L'$. Assume that $l'_v\leq 0$ for all
$v\in \calv(|Z|)$ and for any connected component $Z_{con}$ of $Z$ there exists
$v\in \calv$ adjacent with $Z_{con}$ with $l'_v<0$.
 (The adjacent condition is $|Z_{con}|\cap E_v \not=\emptyset$.)
 Then the conditions from Theorem \ref{th:CohNat} are satisfied, hence  $h^i(Z,\calO_Z(l'))=h^i(Z,\calL_{gen})$ and (\ref{eq:cohNat}) holds.

Indeed, first note that if for some vertex with $l'_v=0$ one has $(l',E_v)\geq 0$ then
$l'_u=0$  for all adjacent vertices $u$ of $v$. Hence, $(l',E_v)\geq 0$ for all vertices $v$ with
$l'_v=0$ contradicts the assumption. That is, there exists $v\in \calv(|Z|)$ so that $l'_v=0$ and $(l',E_v)<0$.

Then we construct the computation sequence as follows.
  At the first part of the computation sequence, at step
 $(l'_k,Z_k)$ we choose $E_{w(k)}$ so that $E_{w(k)}\subset |Z_k|$, the $E_{w(k)}$--coefficient of
 $l'_k$ is zero, and $(E_{w(k)},l'_k)<0$. After finitely many such steps we arrive to the situation when
 along the support of $Z_{k'}$ all the coefficients of $l'_{k'}$ will be strict negative. Then we can
 continue the algorithm arbitrarily.
\end{example}

\begin{corollary}\label{cor:h1Z}
If $\tX$ is generic with dual graph $\Gamma$  and  $|Z|$ is connected   then
\begin{equation}\label{eq:h1Z}
h^1(\calO_Z) = 1-\min_{0< l \leq Z,l\in L}\{\chi(l)\}=
 1-\min_{|Z|\leq l \leq Z,l\in L}\{\chi(l)\}.
\end{equation}

\end{corollary}
\begin{proof} For $D=|Z|$ or $D=E_v$  for any $E_v\subset |Z|$ one has
\begin{equation}
0 \to H^0(Z-D,\calO_Z(-D)) \to H^0(\calO_Z) \stackrel{\delta}{\to} H^0(\calO_D)
\to  H^1(Z-D, \calO_Z(-D)) \stackrel{\iota}{ \to} H^1(\calO_Z) \to 0.
\end{equation}
Since $\delta$ is onto $\iota$ is an isomorphism.
 But for $ h^1(Z-D, \calO_Z(-D))$ Example \ref{ex:cohNat} and (\ref{eq:cohNat}) hold.
\end{proof}

\subsection{The cohomology of natural line bundles over $\tX$.}\label{ss:CoHNatX}
Next we apply the
 results of the previous subsection for a cycle $Z$ with all its coefficients very large.
 Recall that by Artin's Criterion $p_g=0$ (that is, $(X,o)$ is rational) if and only if
 $\min _{l\in L_{>0}}\{\chi(l)\}=1$ \cite{Artin62,Artin66}. Furthermore, for any singularity
  $\min _{l\in L_{\geq 0}}\{\chi(l)\}=\min _{l\in L}\{\chi(l)\}$, see e.g.
   \cite[Prop. 4.3.3]{trieste}.

 \begin{corollary}\label{cor:pg}
\begin{equation}\label{eq:pg}
p_g(X,o)= 1-\min_{l\in L_{>0}}\{\chi(l)\} =-\min_{l\in L}\{\chi(l)\}+\begin{cases}
1 & \mbox{if $(X,o)$ is not rational}, \\
0 & \mbox{else}.
\end{cases}
\end{equation}
\end{corollary}
\begin{proof}For the first identity use (\ref{eq:h1Z}), for the second one
 use Artin's Criterion for rationality.
\end{proof}

\begin{remark}\label{rem:Wa} (a)
For {\it any } non--rational  analytic structure $(X,o)$ one has $p_g(X,o)\geq 1-\min_{l\in L}\{\chi(l)\}$ \cite{Wa70,NO17}.
The above corollary shows that this topological bound in fact is optimal.

(b) If $(X,o)$ is elliptic then $\min _{l\in L_{>0}}\{\chi(l)\}=0$.
Hence,  if the analytic structure is generic then $p_g=
1-\min _{l\in L_{>0}}\{\chi(l)\}=1$. This was proved by Laufer in \cite{Laufer77}.
\end{remark}
\begin{corollary}\label{cor:CohNat} Assume that $\tX$  is generic with
dual graph $\Gamma$. Choose
 any  $l'\in L'$ and consider  $\calO_{\tX}(l')$,  the natural line bundle on $\tX$.
 Then
\begin{equation}\label{eq:CohNatCor}
h^1(\tX,\calO_{\tX}(l'))=\chi(-l')-\min _{l\in L_{\geq 0}}\, \{\chi(-l'+l)\}+\epsilon(l'),
\end{equation}
where
\begin{equation*}
\epsilon(l')=\begin{cases}
1 & \mbox{if \ $l'\in L$, $l'\geq 0$, and $(X,o)$ is not rational}, \\
0 & \mbox{else}.
\end{cases}
\end{equation*}
\end{corollary}
 \begin{proof}
For any effective cycle $Z$ (with $|Z|=E$) and $l'\in L'$ let us write
$\Delta(Z,l'):= h^1(Z,\calO_Z(l'))-\chi(-l')+
\min _{0\leq l\leq Z, l\in L}\, \{\chi(-l'+ l)\}$.
 In order to compute $h^1(\tX,\calO_{\tX}(l'))$ let us
 fix some $Z$ with all its coefficients very large.
Then, if we start with the pair $(l',Z)$,
   the Laufer sequence from \ref{bek:Lauferseq} ends with some
 $(l'_t,Z_t)$ with $Z_t\geq E$ (still with large coefficients), and $-l'_t\in \calS'$.
  We claim that $\Delta(Z_k,l'_k)$ is constant along the computation sequence. Indeed, from the
 cohomological exact sequence used in \ref{bek:Lauferseq} (for $k=0$)
$h^1(Z,\calO(l'))=h^1(Z-E_w,\calO(l'-E_w))-1-(E_w,l')$. Then, we compare
$\min_{0\leq l\leq Z}\chi(-l'+l) $ and $\min_{0\leq l\leq Z-E_w}\chi(-l'+E_w+l) $.
Since for any $x\geq 0$ with $E_w\not\in|x|$ we have
$\chi(-l'+E_w+x)\leq \chi(-l'+x)$, these two minima agree. Hence the claim  follows.

Now, for the pair $(l'_t,Z_t)$, with $-l'_t\in \calS'$,  we distinguish two cases.
The case $l'_t=0$ occurs exactly when $l'\in L_{\geq 0}$
(because $l'_t$ is the largest element of $(-\calS')\cap (l'-L_{\geq 0})$, cf.
\cite[Prop. 4.3.3]{trieste}). In this case $\Delta(Z_t,l'_t)$
can be computed from (\ref{eq:pg}). Or,
$l'_t\not=0$. In this case all the coefficients of $l'_t$ are strict negative
 (use e.g. Remark \ref{rem:cohGen}(c)), and  $\Delta(Z_t,l'_t)=0$ by (\ref{eq:cohNat}).
 \end{proof}

\begin{example}\label{ex:CohNat} For any $h\in H$ define $k_h:=K+2r_h$ and
$$\chi_{k_h}(x):=
 -(x,x+k_h)/2=\chi(x)-(x,r_h)=\chi(x+r_h)-\chi(r_h).$$
 (For the definition of $r_h$ see \ref{ss:2.1}.)
  It is known (use e.g. the algorithm from \cite[Prop. 4.3.3]{trieste}) that
 for any $h\in H$ one has $\min _{l\in L_{\geq 0}}\, \chi(r_h+l)=
 \min _{l\in L}\, \chi(r_h+l)$. Therefore, for $h\not=0$ one has
\begin{equation}\label{eq:rh}
h^1(\tX, \calO_{\tX}(-r_h))=\chi(r_h)-\min _{l\in L}\, \chi(r_h+l)=- \min _{ l\in L}
\{\chi_{k_h}(l)\}=- \min _{ l\in L_{\geq 0}}
\{\chi_{k_h}(l)\}.
\end{equation}
\end{example}
\begin{remark}\label{rem:pgUAC} (a)
Let $(X_{ab},o)$ be the {\bf universal abelian covering} of $(X,o)$.
Then
$$p_g(X_{ab},0)=\sum _{h\in H}h^1(\tX, \calO_{\tX}(-r_h)),$$
see e.g. \cite{trieste}.
Hence $p_g(X_{ab},0)$ is topologically (and explicitly)  computable by (\ref{eq:pg}) and (\ref{eq:rh}).

(b) For a conjectural identity which connects
$\min _{l\in L}\, \chi(r_h+l)$ with the
Heegaard Floer
$d$--invariant associated with the link of the singularity and the $spin^c$--structure
attached to the characteristic element $k_h$ see \cite[\S 5.2]{lattice}.
\end{remark}

\subsection{The cohomological cycle of $\tX$}\label{ss:cohcycle}
For any non--rational germ and fixed resolution the set
$ \{Z\in L_{>0}\,:\, h^1(\calO_Z)=p_g(X,o)\}$ has a unique minimal element $Z_{coh}$,
called the cohomological cycle. It also satisfies the next  property: $h^1(\calO_Z)<p_g$
for any $Z\not\geq Z_{coh}$, $Z>0$ (see e.g. \cite[4.8]{MR}).

In parallel, let us mention the following topological statement. For any
fixed non--rational resolution graph,
 $\calm :=\{ Z\in L_{>0}: \chi(Z)=\min_{l\in L} \chi(l)\}$ has a unique minimal and a unique maximal element.
Indeed, if $l_1,l_2\in \calm$, then for $m:=\min \{l_1,l_2\}$ and
$M:=\max \{l_1,l_2\}$ one has
$\chi(M)+\chi(m)= \chi(l_1)+\chi(l_2)-(l_1-m,l_2-m)\leq 2\min\chi$, hence $\chi(m)=\chi(M)=
\min \chi$. Hence, $M\in \calm$ always, and $m\in \calm$ whenever  $m\not=0$. However,
if $m=0$ then the germ is elliptic and $\calm$ admits a minimal element, namely the minimally elliptic cycle \cite{Laufer77,weakly,Nfive}.

\begin{corollary}\label{cor:cohcycle}
Assume that $\tX$ is generic with a non--rational dual graph $\Gamma$. Then the
cohomological cycle   $Z_{coh}:=\min\{Z\in L_{>0}\,:\, h^1(\calO_Z)=p_g(X,o)\}$, is
$\min \{ Z\in L_{>0}\,:\, \chi(Z)=\min_{l\in L} \chi(l)\}$.
\end{corollary}

\subsection{The cohomological cycle of a line bundle}\label{ss:cohcycle2}
For any $\calL\in \pic(\widetilde{X})$
with $h^1(\widetilde{X},\calL)>0$ the set
$L_{\calL}:=\{l\in L_{>0}\, : \, h^1(l,\calL)=h^1(\widetilde{X},\calL)\}$
has a unique minimal element, denoted by $Z_{coh}(\calL)$, called
the cohomological cycle of $\calL$ (and of $\phi$). Similarly,
for any  $Z>0$ and $\calL\in \pic(Z)$
with $h^1(Z,\calL)>0$ the set
$L_{Z,\calL}:=\{l\in L, \ 0<l\leq Z\, : \, h^1(l,\calL)=h^1(Z,\calL)\}$
has a unique minimal element, denoted by $Z_{coh}(Z,\calL)$, called
the cohomological cycle of $(Z,\calL)$. (For detail see e.g. \cite[5.5]{NNI}.)

\begin{corollary}\label{cor:cohcyc2} Assume that $\tX$ is generic.

(a) Fix any $l'\in L'$ with $h^1(\tX, \calO_{\tX}(l'))\not=0$. Then the set
$$L_{l'}:=\{l_{min}\in L_{\geq 0}\ |\ \chi(-l'+l_{min})=
 \min_{l\in L_{\geq 0}}\chi( -l'+l)\}$$
 has a unique minimal element $Z_{coh}(l')$, which coincides with the
 cohomological cycle of  $\calO_{\tX}(l')$.

 (b) For any $Z>0$ and $l'\in L'$  with $h^1(Z, \calO_{\tX}(l'))\not=0$ the set
$$L_{Z,l'}:=\{l_{min}\in L, \ 0\leq l_{min}\leq Z, \ |\ \chi(-l'+l_{min})=
 \min_{0\leq l\leq Z,\, l\in L}\chi( -l'+l)\}.$$
has a unique minimal element $Z_{coh}(Z,l')$, which coincides with the
 cohomological cycle of $\calO_{\tX}(l')|_Z$.
\end{corollary}
\begin{remark}\label{rem:cohcyc}\ \cite[5.5]{NNI}
For any analytic structure $(X,o)$ supported on the fixed topological type and for any
resolution $\phi$, fix $l'$ such that for the generic line bundle $\calL_{gen}\in\pic^{l'}(\tX)$
one has  $h^1(\tX,\calL_{gen})\not=0$. Then the cohomology cycle of $\calL_{gen}$ is
$Z_{coh}(l')$ (independently of the analytic structure). Similarly, if $h^1(Z,\calL_{gen})\not=0$
for the generic $\calL_{gen}\in\pic^{l'}(Z)$ then the cohomological cycle of the pair $(Z,\calL_{gen})$
is $Z_{coh}(Z,l')$.

\end{remark}

\subsection{The Hilbert series}\label{ss:Hilbert}
Fix $\tX$ generic  and let $H(\bt)$ be the multivariable
(equivariant) Hilbert
series associated with the divisorial filtration of the local algebra of the universal abelian covering
of $(X,o)$ associated with divisors supported on all irreducible exceptional   divisors of $\tX$;
for details see e.g. \cite{CDGPs,CDGEq,NCL}. Write $H(\bt)=\sum _{l'\in L'} \mathfrak{h}(l')\bt^{l'}$.
(Here if $l'=\sum_v l'_vE_v$ then $\bt^{l'}=\prod_v t_v^{l'_v}$.) It is known that for any $l'$ there exists a unique $s(l')\in \calS'$ such that $s(l')-l'\in L_{\geq 0}$, and $s(l')$ is minimal with these properties. Furthermore,  for any $l'\in L'$ one has $\mathfrak{h}(l')=\mathfrak{h}(s(l'))$.
Hence it is enough to determine $\mathfrak{h}(l')$ for the (closed) first quadrant
 (because $\calS'\subset L'_{\geq 0}$).

Write $l'$ as $r_h+l_0$ for some $l_0\in L_{\geq 0}$ (and $h=[l']$).
Recall that
$\mathfrak{h}(l')$ is the dimension of $H^0(\calO_{\tX}(-r_h))/H^0(\calO_{\tX}(-l_0-r_h))$, see e.g.
\cite[(2.3.3)]{NCL}. Therefore, for $l_0=0$ we get $\mathfrak{h}(r_h)=0$.

\begin{proposition}\label{prop:Hilbert} Assume that $l'=r_h+l_0$ with $l_0>0$. Then
for $h\not=0$
\begin{equation}\label{eq:Hilbert}
\mathfrak{h}(l')=
\min_{l\in L_{\geq 0}} \{\chi(l'+l)\}-\min_{l\in L_{\geq 0}} \{\chi(r_h+l)\}=
\min_{l\in L_{\geq 0}} \{\chi_{k_h}(l_0+l)\}-\min_{l\in L_{\geq 0}} \{\chi_{k_h}(l)\}.
\end{equation}
For $h=0$ (i.e. when $r_h=0$ and $l'=l_0>0$)
\begin{equation}\label{eq:Hilbert2}
\mathfrak{h}(l_0)=\min_{l\in L_{\geq 0}} \{\chi(l_0+l)\}-\min_{l\in L_{\geq 0}} \{\chi(l)\}+
\begin{cases}
1 & \mbox{if $(X,o)$ is not rational}, \\
0 & \mbox{else}.
\end{cases}
\end{equation}
\end{proposition}
\begin{proof}
Use the exact sequence $0\to \calO(-r_h-l_0)\to \calO(-r_h)\to \calO_{l_0}(-r_h)\to 0$ and Corollary
\ref{cor:CohNat}.
\end{proof}
\begin{remark} Proposition \ref{prop:Hilbert}  via (\ref{eq:rh}) and Corollary \ref{cor:pg}
 can be written $h$--uniformly:
 $$\mathfrak{h}(r_h+l_0)=
 \min_{l\in L_{\geq 0}} \{\chi_{k_h}(l_0+l)\}+h^1(\tX,\calO_{\tX}(-r_h))
 \ \ \ (\forall\ h\in H, \ l_0\in L_{>0}).$$
\end{remark}

\subsection{The Poincar\'e series} \label{ss:Poncare}
Let $P(\bt)$ be the multivariable equivariant Poincar\'e series associated with $(X,o)$ and its
fixed resolution, cf. \cite{CDGPs,CDGEq,NCL}. It is defined as $P(\bt)=-H(\bt)\cdot \prod_{v\in\calv}
(1-t_v^{-1})$. It is known that it is  supported on $\calS'$.  Proposition
\ref{prop:Hilbert} implies the following.
\begin{corollary}\label{cor:Poincare} Write $P(\bt)=\sum_{l'\in\calS'}\mathfrak{p}(l')\bt^{l'}$. Then $\mathfrak{p}(0)=1$ and for $l'>0$ one has
$$\mathfrak{p}(l')= \sum_{I\subset \calv} (-1)^{|I|+1}\, \min_{ l\in L_{\geq 0}}
\chi(l'+l+E_I).$$
\end{corollary}

\subsection{The analytic semigroup}\label{ss:AnnalSemigr}
The analytic semigroup is defined as
$$\calS'_{an}:= \{l'\,:\, H^0(\tX,\calO_{\tX}(l'))_{reg}\not=\emptyset\}= \{l'\,:\,
\mathfrak{h}(l')<\mathfrak{h}(l'+E_v) \
 \mbox{for any $v\in\calv$}\}.$$
\begin{corollary}\label{cor:ansemgr}
If $(X,o)$ is generic then
$\calS'_{an}= \{l'\,:\, \chi(l')<
\chi(l' +l) \ \mbox{for any $l\in L_{>0}$}\}\cup\{0\}$ and $h^1(\tX, \calO_{\tX}(l'))=0$ for any
$l' \in-\calS'_{an}\setminus \{0\}$.
\end{corollary}
\begin{proof}
Use  Corollary \ref{cor:CohNat} and Proposition \ref{prop:Hilbert}.
\end{proof}
\begin{remark}\label{rem:ab} (a)
This formula emphasizes once more the parallelism between generic line bundles (associated with
an arbitrary analytic structure) and the natural line bundles associated with a generic analytic structure, cf. \ref{rem:natisgen} and \ref{rem:cohcyc}.
To explain this in the present situation,
consider first an arbitrary analytic structure, a resolution with fixed graph $\Gamma$,
and an effective cycle $|Z|$ as usual.
 By \cite[\S 4]{NNI} the fact that the Abel map $c^{l'}: \eca^{l'}(Z)\to \pic^{l'}(Z)$ is dominant
is independent of the analytic structure,
and it has a purely combinatorial description:
$\chi(-l')< \chi(-l'+l) \ \mbox{for any $l\in L$, $0<  l\leq Z$}\}$. Assume that $Z\gg 0$ and $l'\not=0$.
Then a generic line bundle $\calL_{gen}\in \pic^{l'}(Z)$ is in
${\rm im}(c^{l'})$
if and only if $-l'\in \calS'_{dom}:=\{-l'\,:\,
\chi(-l')< \chi(-l'+l) \ \mbox{for any $l\in L_{>0}$}\}$. On the other hand,
by Corollary \ref{cor:ansemgr},
in the context of a generic analytic type,
this happens exactly when the natural line $\calO_{Z}(l')$ is in the image of $ {\rm im}(c^{l'})$
(that is,  $\calO_{Z}(l')$ behaves as a generic line bundle).
In particular, for generic $\tX$,
$\calS'_{an}=\calS'_{dom}\cup\{0\}$.

(b) In \cite[\S  4]{NNI} several combinatorial properties of $\calS'_{dom}$ are listed.

(c) Corollary \ref{cor:ansemgr} can be compared with the definition of 
$\calS'=\{l' \,:\, \chi(l')<\chi(l'+E_v) \ \mbox{for any $v\in\calv$}\}$.
\end{remark}
\bekezdes $\calS_{an} :=\calS'_{an}\cap L$ is the semigroup of divisors
(restricted to $E$) of functions $\phi^* \calO_{(X,o)}$.  Let $Z_{max}$ be the {\bf
maximal ideal cycle} (of S. S.-T. Yau \cite{Yau1}), that is, the divisorial part of
$\phi^*(\mathfrak{m}_{(X,o)})$ (here $\mathfrak{m}_{(X,o)}$ is the maximal ideal of $\calO_{(X,o)}$).
It is the unique smallest nonzero element of $\calS_{an}$.

\begin{corollary}\label{cor:maxidealcycle} Assume that $\tX$ is generic with non--rational graph $\Gamma$. Then
$\calm=\{ Z\in L_{>0}\,:\, \chi(Z)=\min _{l\in L}\chi(l)\}$ has a unique maximal element and
$Z_{max}=\max\calm$.
\end{corollary}
\begin{proof}
For the first part see the second paragraph of \ref{ss:cohcycle}.
$\max\calm \in \calS_{an}$ by the right hand side of \ref{cor:ansemgr}, but
$\min\calS_{an}$ cannot be smaller than $\max\calm$ by the very same identity.
\end{proof}
\begin{remark}
Recall that the fundamental (or minimal, or Artin) cycle $Z_{min}:=\min \{\calS'\cap L_{>0}\}$
has the property $h^0(\calO_{Z_{min}})=1$, hence $h^1(\calO_{Z_{min}})=1-\chi(Z_{min})$
(see e.g. \cite{Nfive}). For $\tX$ generic and $(X,o)$  non--rational
  any cycle $Z\in\calm$
(in particular $Z_{max}$ too) has this  property. Indeed, $h^1(\calO_Z)=1-\min _{0<l\leq Z}\chi(l)
=1-\chi(Z) $, hence $h^0(\calO_Z)=1$ too.
\end{remark}

\begin{corollary}\label{cor:cyccyc}
For $(X,o)$ generic one has $Z_{max}\geq Z_{coh}$. If additionally $(X,o)$ is numerically
Gorenstein then $Z_{coh}+Z_{max}=Z_K$.
\end{corollary}
\subsection{The $\calO_{(X,o)}$--multiplication on $H^1(\tX,\calO_{\tX})$}\label{ss:module}
Assume that $p_g>0$. On $H^1(\tX,\calO_{\tX})$ the $\calO_{(X,o)}$--module  multiplication
 transforms on the dual vector space  $H^1(\tX,\calO_{\tX})^*=H^0(\tX\setminus E, \Omega^2_{\tX})/H^0(\tX,\Omega^2_{\tX})$
into the multiplication of forms by functions.
The filtration on $H^1(\tX,\calO_{\tX})$  induced by the powers of the maximal ideal agrees with the filtration associated   by the nilpotent operator determined  by multiplication by a generic element of $\mathfrak{m}_{(X,o)}$. For details see e.g.  \cite{Tomari86}.

The poles of forms are bounded by $Z_{coh}$.
Indeed, by the exact sequence $0\to \Omega^2\to \Omega^2(Z_{coh})\to \calO_{Z_{coh}}(Z_{coh}+K_{\tX})\to
0$ and from the vanishing $h^1(\Omega^2)=0$ (and from Serre duality) we have
$\dim H^0(\Omega^2(Z_{coh}))/H^0(\Omega^2)=h^0(\calO_{Z_{coh}}(Z_{coh}+K_{\tX}))=h^1(\calO_{Z_{coh}})=p_g$.
 Hence the subspace $H^0(\Omega^2(Z_{coh}))/H^0(\Omega^2)
 \subset H^0(\tX\setminus E,\Omega^2)/H^0(\Omega^2)$ has codimension zero, hence the spaces agree.

 \begin{corollary}\label{cor:poles}
 If $\tX$ is generic then $\mathfrak{m}_{(X,o)}\cdot H^1(\tX,\calO_{\tX})=0$.
 In particular, the $\calO_{(X,o)} $--module multiplication factorizes to the $\C=\calO_{(X,o)}/
 \mathfrak{m}_{(X,o)}$--vector space structure.
 \end{corollary}
\begin{proof}
Since $Z_{max}\geq Z_{coh}$, cf. \ref{cor:cyccyc},
$\mathfrak{m}_{(X,o)}\cdot H^0(\Omega^2(Z_{coh}))\subset H^0(\Omega^2(-Z_{max}+Z_{coh}))\subset
H^0(\Omega^2)$.
\end{proof}

\subsection{Generic $\Q$--Gorenstein singularities}\label{ss:Goren}
Recall that a singularity $(X,o)$ is Gorenstein if
the anticanonical cycle $Z_K$ is integral, and
$\Omega^2_{\tX}=\calO_{\tX}(K_{\tX})$ equals
$\calO_{\tX}(-Z_K)$.  Hence in this case
$\calO_{\tX}(K_{\tX})$ is natural. Recall, that more generally,
 a line bunlde $\calL$ is natural if and only if
one of its powers has the form $\calO_{\tX}(l)$ for some $l\in L$, or equivalently,
if and only if
its restriction $\calL|_{\tX\setminus E}\in \pic(\tX\setminus E)={\rm Cl}(X,o)$ has finite order
(that is, it is $\Q$--Cartier). In particular, $(X,o)$ is $\Q$--Gorenstein if and only if
$\calO_{\tX}(K_{\tX})$ is a natural line bundle, which automatically should agree with $\calO_{\tX}(-Z_K)$.

\begin{proposition}\label{prop:QGor}
If a $\Q$--Gorenstein singularity $(X,o)$ admits a resolution $\tX$
 with generic analytic structure, then $(X,o)$ is either rational of minimally elliptic.
\end{proposition}
\begin{proof} 
{\bf Step 1.} Let us fix a resolution $\tX$  of a normal surface singularity $(X,o)$.
We claim that if $(X,o)$ is neither rational nor minimally elliptic then there exists an
effective cycle $Z>0$, $|Z|\subset E$, with $Z\not\geq Z_K$ and with $h^1(\calO_{Z})>0$.

Assume first that $\tX=\tX_{min}$ is a minimal resolution. Then $Z_K\geq 0$ (by adjunction formulae, see also
\cite{Laumult}).
By vanishing  $h^1(\calO_{\tX}(-\lfloor Z_K \rfloor))=0$ we get that
$h^1(\calO_{\lfloor Z_K \rfloor})=p_g$.   Since $(X,o)$
is not rational, necessarily  $\lfloor Z_K \rfloor>0$. Hence, if
$\lfloor Z_K \rfloor<Z_K$ then $Z=\lfloor Z_K \rfloor$ works.

Assume that
$\lfloor Z_K \rfloor =Z_K$. Then $Z_K\in L$ and  $Z_K>0$ (since $p_g>0$)
hence  necessarily  $Z_K\geq E$ (see \cite{Laumult}).  For any $v\in\calv$ consider the exacts sequence
$ 0\to \calO_{E_v}(-Z_K+E_v)\to \calO_{Z_K}\to \calO_{Z_K-E_v}\to 0$.
If  $h^1(\calO_{Z_K-E_v})>0$ for some $v$ then we take  $Z=Z_K-E_v$.
Otherwise, $h^1(\calO_{Z_K-E_v})=0$ for every $v$. Since $h^1(\calO_{E_v}(-Z_K+E_v))=1$
we get that $p_g=1$ and $Z_K=Z_{coh}$. Then the geometric genus of the singularities
 obtained by contracting any $E\setminus E_v$ is rational, hence $(X,o)$ is minimally elliptic
(for details see \cite{Laufer77} or \cite{MR}).

Finally, let $\tX$ be arbitrary and let $\pi:\tX\to \tX_{min}$ be the corresponding modification of the minimal one. Let $0<Z<Z_K$ be the cycle obtained previously for $\tX_{min}$.
Then $\pi^*(Z)$ works in $\tX$.

{\bf Step2.} Fix the generic resolution $\tX$.
Assume that $(X,o)$ is neither rational nor minimally elliptic. Chose a cycle $Z$
as in Step 1. Using $0\to \Omega^2_{\tX}\to \Omega^2_{\tX}(Z)\to \calO_{Z}(Z+K_{\tX})\to 0$, we get that
$h^1(\Omega^2_{\tX}(Z))=h^1(\calO_Z(Z+K_{\tX}))=h^0(\calO_Z)$. Since $(X,o)$ is $\Q$--Gorenstein,
$\Omega^2_{\tX}(Z)=\calO_{\tX}(Z-Z_K)$, hence
$h^1(\calO_{\tX}(Z-Z_K))=h^0(\calO_Z)=
\chi(Z)+h^1(\calO_Z)$.
Now we apply (\ref{eq:CohNatCor})  and (\ref{eq:h1Z}), and we get
 $$\chi(Z_K-Z)-\min_{l\geq 0}\{\chi (Z_K-Z+l)\}= \chi(Z)+1-\min_{0<l\leq Z}\{\chi(l)\}.$$
Since $\chi(D)=\chi(Z_K-D)$ this transforms into
 $-\min_{l\leq Z}\{\chi (l)\}= 1-\min_{0<l\leq Z}\{\chi(l)\}$.
 Next we claim that $\min_{l\leq Z}\{\chi (l)\}=\min_{0\leq l\leq Z}\{\chi (l)\}$.
 Indeed, if $l=l_+-l_-$ with $l_+, l_-\geq 0$ and with different supports, then
 there exists $E_v\in |l_-|$ such that $(E_v,l_-)<0$; then by a computation
 $\chi(l+E_v)\leq \chi(l)$. Hence inductively $\chi(l_+)\leq \chi(l)$. Therefore,
 $$-\min_{0\leq l\leq Z}\{\chi (l)\}= 1-\min_{0<l\leq Z}\{\chi(l)\}.$$
 This means that $\min_{0\leq l\leq Z}\{\chi (l)\}$ cannot be realized by an element $l>0$, hence
 $0=\chi(0)< \min_{0<l\leq Z}\{\chi(l)\}$. But this implies $h^1(\calO_Z)=0$ (see \cite[Example 4.1.3]{NNI}), a contradiction.
\end{proof}

\begin{remark}\label{rem:Laufergen}
Proposition \ref{prop:QGor} generalizes the following result of Laufer
 \cite[Th. 4.3]{Laufer77} (with a different proof):
if the generic analytic structure of a numerically Gorenstein topological type  is Gorenstein then the topological type is either Klein or minimally elliptic.
(Recall that the Klein --- or $ADE$ ---
singularities are exactly the Gorenstein rational singularities.)
\end{remark}


\begin{thebibliography}{30}

\bibitem[ACGH85]{ACGH}  Arbarello, E., Cornalba, M., Griffiths, P. A.,
 Harris, J.: Geometry of algebraic curves,
 Vol. I. Grundlehren der Mathematischen Wissenschaften 267, Springer
Verlag, New York, 1985.


\bibitem[A62]{Artin62} Artin, M.:
Some numerical criteria for contractibility of curves on algebraic surfaces.
{\em  Amer. J. of Math.}, {\bf 84}, 485-496, 1962.

\bibitem[A66]{Artin66} Artin, M.:
On isolated rational singularities of surfaces.
{\em Amer. J. of Math.}, {\bf 88}, 129-136, 1966.






\bibitem[CDGZ04]{CDGPs} Campillo, A.,  Delgado, F. and Gusein-Zade, S. M.:
Poincar\'e series of a rational surface singularity, {\em Invent. Math.} {\bf 155} (2004),
no. 1, 41--53.

\bibitem[CDGZ08]{CDGEq}  Campillo, A.,  Delgado, F. and Gusein-Zade, S. M.:
Universal abelian covers of rational
surface singularities and multi-index filtrations,
{\em Funk. Anal. i Prilozhen.} {\bf 42} (2008), no. 2, 3--10.









\bibitem[Fl10]{Flamini} Flamini, F.: Lectures on Brill--Noether theory,
 in Proceedings of the workshop "Curves and Jacobians", Eds. J-M Muk, Y. R. Kim, Korea Institute for Advanced Study, (2011), 1-20.


\bibitem[GR62]{GRa} Grauert, H.: \"Uber Modifikationen und exzeptionelle
analytische Mengen,   {\it Math. Ann.} {\bf 146} (1962), 331--368.



\bibitem[GrRie70]{GrRie} Grauert, H. and Riemenschneider, O.:
Verschwindungss\"atze f\"ur analytische
kohomologiegruppen auf komplexen R\"aumen, {\it Inventiones math.}
{\bf 11} (1970), 263--292.

\bibitem[Gro62]{Groth62} Grothendieck, A.: Fondements de la g\'eom\'etrie
alg\'ebrique, [Extraits du S\'eminaire Bourbaki 1957--1962], Secr\'etariat
math\'ematique, Paris 1962.




\bibitem[Kl05]{Kl} Kleiman, St. L.: The Picard scheme, in
`Fundamental Algebraic Geometry: Grothendieck’s FGA Explained',
Mathematical Surveys and Monographs
Volume: 123; 2005, 248--333.

\bibitem[Kl13]{Kleiman2013} Kleiman, St. L.: The Picard Scheme, In `Alexandre
 Grothendieck: A Mathematical Portrait', International Press of Boston, Inc., 2014
 (L. Schneps editor).


\bibitem[KSB88]{KSB} Koll\'ar, J. and Shepherd-Barron, N.I.: Threefolds and
deformations of surface singularities, {\it Invent. math.}  91
(1988), 299-338.




\bibitem[L13]{LPhd} L\'aszl\'o, T.: Lattice cohomology and
Seiberg--Witten invariants of normal surface singularities, PhD. thesis,
Central European University, Budapest, 2013.




\bibitem[La71]{Lauferbook} Laufer, H.B.: Normal two--dimensional singularities.
{\em Annals of Math. Studies}, {\bf 71}, Princeton University Press, 1971.


\bibitem[La72]{Laufer72} Laufer, H.B.: On rational singularities,
{\em Amer. J. of Math.}, {\bf 94}, 597-608, 1972.

\bibitem[La73]{LaDef1}
Laufer, Henry B.: Deformations of Resolutions of Two-Dimensional Singularities,
 {\it Rice Institute Pamphlet - Rice University Studies} {\bf 59} no. 1 (1973), 53--96.

\bibitem[La73b]{Laufer73} Laufer, H.B.: Taut two--dimensional singularities,
{\em Math. Ann.}, {\bf 205}, 131-164,  1973.

\bibitem[La77]{Laufer77} Laufer, H.B.: On minimally elliptic singularities,
{\em Amer. J. of Math.} {\bf 99} (1977), 1257--1295.

\bibitem[La83]{LaW} Laufer, H.B.: Weak simultaneous resolution for deformations
of Gorenstein surface singularities, {\em Proc. of Symp. in Pure Math.},
{\bf 40}, Part 2 (1983), 1-29.

\bibitem[La87]{Laumult} Laufer, H.B.:
The multiplicity of isolated two--dimensional
hypersurface singularities, {\em Transactions of the AMS}, {\bf 302}
Number 2, 489-496, 1987.



\bibitem[Li69]{Lipman} Lipman, J.: Rational singularities,
 with applications to algebraic surfaces
and unique factorization, {\it Inst. Hautes \'Etudes Sci. Publ. Math.} {\bf 36} (1969), 195-279.




\bibitem[Mu66]{MumfordCurves} Mumford, D.: Lectures on curves on an algebraic surface,
{\it Ann. of Math. Studies} {\bf 59}, Princeton, 1966.



\bibitem[NN18]{NNI} Nagy, J., N\'emethi, A.:
The Abel map for surface singularities  I. Generalities and  examples,
arXiv:1809.03737, to appear in {\it Math. Annalen}.

\bibitem[N99]{weakly} N\'emethi, A.: ``Weakly'' Elliptic Gorenstein
singularities of surfaces,
{\em Inventiones math.},  {\bf 137}, 145-167 (1999).


\bibitem[N99b]{Nfive} N\'emethi, A.: Five lectures on normal surface singularities,
lectures at the Summer School in {\em Low dimensional topology} Budapest,
Hungary, 1998; Bolyai Society Math. Studies {\bf 8} (1999), 269--351.


\bibitem[N07]{trieste} N\'emethi, A.: Graded roots and singularities,
{\em Singularities in geometry and topology},  World
Sci. Publ., Hackensack, NJ (2007), 394--463.

\bibitem[N08]{NPS} N\'emethi, A.: Poincar\'e series associated with
surface singularities, in Singularities I, 271--297,
{\em Contemp. Math.} {\bf 474}, Amer. Math. Soc., Providence RI, 2008.

\bibitem[N08b]{lattice} N\'emethi, A.:  Lattice cohomology of normal surface singularities
 {\em Publ. RIMS. Kyoto Univ.}, {\bf 44} (2008),  507--543.

\bibitem[N12]{NCL} N\'emethi, A.: The cohomology of line bundles
of splice--quotient singularities,
{\em Advances in Math.} {\bf 229} 4 (2012), 2503--2524.

\bibitem[N11]{NJEMS} N\'emethi, A.: The Seiberg--Witten invariants
of negative definite plumbed 3--manifolds,
{\em J. Eur. Math. Soc.} {\bf 13} (2011), 959--974.






\bibitem[NO17]{NO17} N\'emethi, A. and Okuma, T.:
Analytic singularities supported by a specific
integral homology sphere link, arXiv:1711.03384;
to appear in the
Proceedings dedicated to H. Laufer's 70th birthday (Conference at Sanya).

















\bibitem[O04]{OkumaRat} Okuma, T.: Universal abelian covers of rational surface singularities,
{\it Journal of London Math. Soc.} {\bf 70}(2) (2004), 307-324.

\bibitem[O08]{Ok} Okuma, T.: The geometric genus of splice--quotient singularities,
{\em Trans. Amer. Math. Soc.} {\bf 360} 12 (2008), 6643--6659.






\bibitem[Os]{osserman} Osserman, B.: Notes on cohomology and base change,
https://www.math.ucdavis.edu/~osserman/.








\bibitem[Ri74]{Ri74} Riemenschneider, O.: Bemerkungen zur Deformationstheorie nichtrationaler
Singularit\"aten, {\it Manus. Math.} {\bf 14} (1974), 91--99.

\bibitem[Ri76]{Ri76}  Riemenschneider, O.: Familien komplexer $R$-ume mit streng pseudokonvexer
spezieller Faaser, {\it Coment. math. Helv.}, {\bf 51} (1976), 547--565.


\bibitem[Re97]{MR}  Reid, M.: Chapters on Algebraic Surfaces.
In: Complex Algebraic Geometry,
IAS/Park City Mathematical Series,  Volume {\bf 3}  (J. Koll\'ar editor),
3-159, 1997.





\bibitem[To86]{Tomari86} Tomari, M.: Maximal-Ideal-Adic Filtration on $R^1\psi_*\cO_{\tilde{V}}$
for Normal Two-Dimensional Singularities, {\it Advanced Studies in Pure Math.} {\bf 8} (1986),
Complex Analytic Singularities,  633-647.

\bibitem[V04]{Wim} Veys, W.: Stringy invariants of normal surfaces, {\it J. Algebraic Geom.}
{\bf 13} (2004), 115--141.

\bibitem[Wa70]{Wa70} Wagreich, Ph.: Elliptic singularities of surfaces, {\it Amer. J. of Math.},
{\bf 92} (1970), 419--454.

\bibitem[Wa76]{Wa76} Wahl, M. J.: Equisingular deformations of normal surface singularities,
{\it Ann. of Math.} {\bf 104} (1976), 325--356.


\bibitem[Y80]{Yau1} Yau, S. S.-T.: On maximally elliptic singularities,
{\em Transactions of the AMS}, {\bf 257} Number 2 (1980), 269-329.

\end{thebibliography}
\end{document}